\documentclass[11pt,a4paper]{article}
\usepackage[utf8]{inputenc}
\usepackage{amsmath,amsfonts,amssymb}
\usepackage{mathtools}
\usepackage{enumerate}
\usepackage{bm}
\usepackage{upgreek}
\usepackage{rotating}
\usepackage{multicol}
\usepackage{array,multirow}
\usepackage{textcomp}
\usepackage{graphicx}
\usepackage{caption}
\usepackage{algorithm}
\usepackage{algpseudocode}
\usepackage{url,hyperref}
\usepackage{amsthm}
\usepackage[left=2cm,right=2cm,top=2cm,bottom=2cm]{geometry}

\usepackage{xcolor}
\usepackage{subfig}
\usepackage{float}
\usepackage{longtable}
\usepackage{placeins}
\usepackage{todonotes}
\usepackage{titletoc}  %\ttl
\newlength{\defbaselineskip}
\newcommand{\setlinespacing}[1]%
{\setlength{\baselineskip}{#1 \defbaselineskip}}
\theoremstyle{plain}
\newtheorem{thm}{Theorem}[section]

\newtheorem{lem}[thm]{Lemma}
\newtheorem{prop}[thm]{Proposition}

\theoremstyle{definition}
\newtheorem{defn}{Definition}[section]
\newtheorem{ass}{Assumption}[section]
\newtheorem{rmk}{Remark}[section]

\newcommand{\eps}{\varepsilon}

\DeclareMathOperator*{\esssup}{esssup}

\newcommand{\cP}{\mathcal{P}}

\newcommand{\cL}{\mathcal{L}}
\newcommand{\cT}{\mathcal{T}}

\newcommand{\cB}{\mathcal{B}}
\newcommand{\cA}{\mathcal{A}}
\newcommand{\cS}{\mathcal{S}}

\newcommand{\cF}{\mathcal{F}}

\newcommand{\bE}{\mathbb{E}}
\newcommand{\bF}{\mathbb{F}}
\newcommand{\bP}{\mathbb{P}}
\newcommand{\bR}{\mathbb{R}}
\newcommand{\bQ}{\mathbb{Q}}
\newcommand{\bN}{\mathbb{N}}

\makeatletter\@addtoreset{equation}{section} \makeatother

\allowdisplaybreaks
\begin{document}

	\title{A Policy Iteration Scheme for Semilinear Stochastic Hamilton–Jacobi–Bellman Equations with Exponential Convergence\thanks{This work was partially supported by the National Science and Engineering Research Council of Canada (NSERC).}}
	\author{Hasib Uddin Molla\footnotemark[2] \and Jinniao Qiu\footnotemark[2]}
	\maketitle

\footnotetext[2]{Department of Mathematics \& Statistics, University of Calgary, 2500 University Drive NW, Calgary, AB T2N 1N4, Canada. \textit{E-mail}:  
\texttt{mdhasibuddin.molla@ucalgary.ca} (H.U. Molla), \texttt{jinniao.qiu@ucalgary.ca} (J. Qiu).  }

%	\maketitle 
	%\tableofcontents
\begin{abstract}

This paper is concerned with the non-Markovian stochastic optimal control problems 
in which the value function is a random field characterized by a stochastic Hamilton-Jacobi-Bellman (SHJB) equation. 
When the stochastic integration coefficients are not controlled, the SHJB equation takes a semilinear form, which is subject to computational challenges compared to the Markovian case due to the measurable randomness. 
We introduce a policy-iteration algorithm based on successive 
linearization that reduces the nonlinear SHJB equation to a sequence of linear ones. 
Furthermore, we prove that the resulting approximation sequence converges monotonically to the value function in the mean-square sense with an exponential rate. 
\\[4pt]
\end{abstract}

%%%%%%%%%%%%%%%%%%%%%%%%%%%%%%%%

	\section{Introduction} 
    Let $\big(\Omega, \cF,\bF,\bP \big)$ be a complete filtered probability space satisfying the usual conditions 
    and carrying two independent standard $m$-dimensional Wiener processes $W$ and $\widetilde{W}$ under the probability measure $\bP$. 
    Let $T>0$ be a finite terminal time. 
    We define the sub-filtration $\bF^W=\{\cF^W_t;0\le t \le T\}$ generated by the Brownian motion $W$ and augmented by $\bP$-null sets. 
	    Since $\bF = \{\cF_t;0\le t \le T\}$, clearly, $\cF^W_t \subset \cF_t$ for all $t\in[0,T]$. We also denote by $\cP$ and $\cP^W$ the $\sigma$-algebra of the predictable sets on $\Omega\times [0,T]$ associated with $\{\cF_t;0\le t\le T\}$ and $\{\cF^W_t;0\le t \le T\}$, respectively.

	Let $A\subset \bR^l$ be a nonempty compact set. We restrict attention to feedback controls of the form
    \begin{equation}\label{general_feedback_control_map}
        \alpha_s = \widetilde{\boldsymbol{\alpha}}(s,X_s),
    \end{equation}
    where the feedback control map $\widetilde{\boldsymbol{\alpha}}$ is $\cP^W\times\cB(\bR^d)$-measurable and $A$-valued. 
Denote by $\cA$ the totality of such feedback controls for which, for all $(t,x)\in[0,T)\times \bR^d$, the following SDE admits a unique strong solution $X^{t,x;\alpha}$:
    \begin{equation}\label{SDE-controlledstate}
		\begin{cases}
			dX_s^{t,x;\alpha} &=b^{\alpha_s} (s,X_s^{t,x;\alpha})ds + \sigma(s,X_s^{t,x;\alpha})dW_s + \widetilde{\sigma}(s,X_s^{t,x;\alpha})d\widetilde{W}_s,\; s\in[t,T] , \\
			X_t^{t,x;\alpha} &=x,
	\end{cases}	
    \end{equation}
    where the coefficients $b, \sigma,\widetilde{\sigma}, f$ and $g$ may be random, explicitly depending on $\omega$, and satisfy Assumption \ref{ass:meas_Lipschitz_LinearGrowth}. 
 	We consider the stochastic control problem 
	\begin{equation}\label{control_problem}
	        \esssup_{\alpha \in \cA} J(t,x;\alpha),
	    \end{equation}
that is to maximize the dynamic gain functional
\begin{align}\label{eq:gain_functional}
		J(t,x;\alpha):= \bE^{\bP} \Bigg[\int_t^T f^{\alpha_s}(s,X_s^{t,x;\alpha})ds + G(X_T^{t,x;\alpha})\Big| \cF_t^W\Bigg],
	\end{align}
	for each $(t,x)\in[0,T]\times \bR^d$. 
    For notational simplicity, we suppress the explicit dependence on $\omega$ and write, for instance, $b^a(s,y)$ instead of $b^a(\omega, s,y)$. 
 % By $X^{t,x;\alpha}=\{X^{t,x;\alpha}_s,s\in[t,T]\}$, we denote the $\bF$-adapted solution to SDE \eqref{SDE-controlledstate}, starting from $x$ at time $t$ and governed by the control process $\alpha$. 
 % The conditional expectation in \eqref{eq:gain_functional} with respect to $\cF_t^W$ corresponds to partial information where only $W$ is observable. 
The value function is defined by
		\begin{align}\label{value_fnc}
			v(t,x)=\esssup_{\alpha \in \cA} J(t,x;\alpha), \quad (t,x)\in[0,T]\times \bR^d.
		\end{align}
	  An admissible control $\alpha^*\in \cA$ is called optimal if it achieves the maximum, i.e.,
	\begin{equation}
		J(t,x;\alpha^*)=v(t,x),\quad \text{$\bP$-a.s.,  for all }(t,x)\in [0,T]\times \mathbb R^d.
	\end{equation}

When the coefficients of the control problem are random (exact measurability conditions are specified in Assumption \ref{ass:meas_Lipschitz_LinearGrowth}), the value function $v$ is a random field adapted to the underlying filtration (see \cite{SPeng1992SHJB,JQiu2018ViscosSolSHJBEqn}). 
In such case,  the generalized dynamic programming principle suggests
 that the value function $v$, together with an auxiliary random field $\psi$ as a pair $(v,\psi)$,  satisfies the following stochastic Hamilton-Jacobi-Bellman (SHJB) equation
{\small	\begin{equation}\label{SHJB-main}
		\begin{cases}
			-dv(t,x)& =\left( tr\Big[\frac{1}{2}\big(\sigma\sigma^{\cT}+\widetilde{\sigma}\widetilde{\sigma}^{\cT}\big)D^2_x v+ D_x\psi \sigma^{\cT} \Big]+\sup_{a\in A} \big((b^{a})^{\cT}D_xv+f^{a}\big)\right)(t,x)dt \\
			&\quad - \psi(t,x) dW_t,
                \quad (t,x)\in[0,T)\times \bR^d,\\
			v(T,x) &= G(x), \quad x\in \bR^d,
		\end{cases}
	\end{equation}}
    provided that $v$ and $\psi$ are sufficiently smooth with respect to the spatial variable $x$; see the pioneering work by Peng \cite{SPeng1992SHJB}.  
The verification theorem asserts that, under appropriate regularity conditions, 
a sufficiently smooth solution of the SHJB equation \eqref{SHJB-main} coincides with the value function of the stochastic optimal control problem 
\eqref{control_problem}. Moreover, if the Hamiltonian
\begin{equation}\label{hamiltonian}
    H(t,x,D_xv,a):=\big((b^{a})^{\cT}D_xv+f^{a}\big)(t,x),
\end{equation}
is such that the map $a\mapsto H(t,x,D_xv,a)$ attains its supremum at $\widetilde{\boldsymbol{\alpha}}(t,x)$ $\bP$-almost surely for each $(t,x)\in[0,T]\times \bR^d$, 
and the corresponding (feedback) control  $\alpha^*$ defined as in \eqref{general_feedback_control_map} belongs to the admissible control set $\cA$, then $\alpha^*$ is an optimal control for the control problem \eqref{control_problem}.
Consequently, solving the SHJB equation is fundamental for identifying both the value function 
and the corresponding optimal control.

\subsection{Main Results}
The SHJB equation in \eqref{SHJB-main} takes 
the form of a semilinear backward stochastic partial differential equation (BSPDE). 
The primary objective of this work is to develop an effective method for
 solving such SHJB equations. To this end, 
 we introduce Algorithm~\ref{alg:SHJB_policy_iteration},
 an iterative scheme for solving the semilinear SHJB equation \eqref{SHJB-main}, 
 in which a linear BSPDE is solved at each iteration. The proposed method follows a policy-improvement strategy, 
 in which the value function is successively refined and can be viewed as a stochastic (non-Markovian)  counterpart 
 of Howard's Policy Improvement algorithm (see \cite{Kerimkulov2020HowardPolicyImprovement}).

\begin{algorithm} [h]
\caption{Policy Improvement Algorithm for SHJB Equation}\label{alg:SHJB_policy_iteration}
	\begin{algorithmic}
		\State \textbf{Initialization:} Choose an initial admissible control $a^0(t,x)\equiv \bar a_0 \in A$ and error tolerance $\eps>0$. Set $n=0$.
        \State \textbf{Step 1 (Solve initial Linear BSPDE):} Compute $(v^0,\psi^0)$ as the solution to the linear BSPDE corresponding to $a^0$:
        \begin{equation}\label{SHJB-Algorithm_step0}
			\begin{cases}
				-dv^0(t,x)& =\left[ tr\big[\frac{1}{2}\big(\sigma \sigma^{\cT}+\widetilde{\sigma}\widetilde{\sigma}^{\cT}\big)D^2_x v^0+ D_x\psi^0\sigma^{\cT} \big]+\big((b^{a^0})^{\cT}D_x v^0+f^{a^0}\big)\right](t,x)dt \\
                &\quad - \psi^0(t,x) dW_t ,\\
				v^0(T,x) &= G(x).
			\end{cases}
		\end{equation}
        \State \textbf{Step 2 (Iterative Improvement):} For $n=1,2,\cdots $ do:
            \begin{enumerate}
                \item \textbf{Policy Improvement:} Define $a^{n}$ by
                \begin{equation}
                    a^{n}(t,x)=\arg\max_{a\in A} \Big(\big((b^a)^{\cT}D_xv^{n-1}+f^a\big)(t,x)\Big).
                \end{equation}
                \item \textbf{Policy Evaluation:} Compute $(v^{n},\psi^{n})$ as the solution to the linear BSPDE with control $a^{n}$:
        \begin{equation}\label{SHJB-Algorithm}
			\begin{cases}
				-dv^{n}(t,x)& =\Big[ tr\big[\frac{1}{2}\big(\sigma \sigma^{\cT}+\widetilde{\sigma}\widetilde{\sigma}^{\cT}\big)D^2_x v^{n}+ D_x\psi^{n}\sigma^{\cT} \big]\\
                &\quad +\big( (b^{a^{n}})^{\cT}D_x v^{n}+f^{a^{n}}\big)\Big](t,x)dt - \psi^{n}(t,x) dW_t ,\\
				v^{n}(T,x) &= G(x).
			\end{cases}
		\end{equation}
                \item \textbf{Stopping Criterion:} If 
                \begin{equation}
                    \sup_{t\in [0,T]} \esssup_{x\in \bR^d} 
                    \bE\Big[\big|v^{n}(t,x)-v^{n-1}(t,x)\big|^2\Big]\leq \eps,
                \end{equation}
                then stop and output $(a^{n},v^{n},\psi^{n})$.
            \end{enumerate}
	\end{algorithmic}
\end{algorithm}

\noindent
The convergence of the sequence $(v^n)_{n\in\bN}$ generated by Algorithm~\ref{alg:SHJB_policy_iteration} toward the value function $v$ is established in the following theorem. The proof is deferred to Section \ref{sec:Theoretical_results_and_proofs}. 

% \begin{align*}
% dX_t &=b^{a^n(t)}(t,X_t)dt + \sigma_t dW_t + \tilde \sigma d\tilde W_t;\\
% -dY^n_t &= f^{a^n(t)}(t,X_t) dt -Z_t dW_t - \tilde Z_t d\tilde W_t;\\
% \end{align*}

\begin{thm}\label{thm:convergence_value_function}
Let Assumptions \ref{ass:meas_Lipschitz_LinearGrowth}--\ref{ass:feedback_control_lipschitz} hold. 
Let the pair $\big(v, \psi\big)\in \big(\cL_{\cP^W}^2(\bP, H^{2,2})\cap \cS_{\cP^W}^2(\bP, H^{1,2})\big)\times \cL_{\cP^W}^2(\bP, H^{1,2})$ 
be the unique strong solution to the SHJB equation \eqref{SHJB-main} and for each $n\in\bN$, 
denote by $\big(v^n, \psi^n\big)\in \big(\cL_{\cP^W}^2(\bP, H^{2,2})\cap \cS_{\cP^W}^2(\bP, H^{1,2})\big)\times \cL_{\cP^W}^2(\bP, H^{1,2})$
 the unique strong solution to the linear SHJB equation \eqref{SHJB-Algorithm}
  at the $n$-th iteration of the algorithm.
Then there exists a constant $q\in (0,1/2)$ depending only on $\theta,\gamma, K$, and $T$ such that, for almost all $(t,x)\in[0,T]\times \bR^d$ and every $n\in\bN$, 
\begin{align}
    \bE \Big[|v^n(t,x)-v(t,x)|^2\Big]\leq C q^n,
\end{align}
where the constant $C$ depends only on $\theta, K, \gamma$, and $T$.
\end{thm}

Furthermore, we establish the monotonicity of the sequence $(v^n)_{n\in\bN}$ generated by Algorithm~\ref{alg:SHJB_policy_iteration} in the subsequent Proposition \ref{prop:monotonicity_value_function}. 
\subsection{Literature,  contributions, and organization of the paper} 
In comparison to the Markovian case where the value function is deterministic and satisfies a deterministic HJB equation (see  \cite{fleming2006controlled,YZ99} for instance), 
the SHJB equations, arising from non-Markovian stochastic controls with random coefficients, are a class of backward stochastic partial differential equations (BSPDEs). In general, when the stochastic-integral coefficients of the controlled state equations 
depend on controls, the resulting SHJB equation is a fully nonlinear BSPDE, while the SHJB equation takes a semilinear form when the stochastic-integral coefficients are independent of controls; see \cite{EnglezosKaratzas09,Hu_Ma_Yong2002SemiLinBSPDE,SPeng1992SHJB,JQiu2018ViscosSolSHJBEqn,YangTang2013SdeRandomCoeff_BSPDE}.
Indeed, the linear and semilinear BSPDEs arise in many applications of probability theory and stochastic processes, 
for instance, in the stochastic control theory for processes with incomplete information, 
as the adjoint equations of the Duncan-Mortensen-Zakai filtering equation (see \cite{Bernsoussan1983MaxPrincDynmProgrmOptimCont, ZhouX1993OptimalControlforSPDEs,TangS1998MaxPrincplePartiallyObservedControlSDE} among many others). 
In the mean-field game systems with common noise, certain classes of BSPDEs are used to characterize the value function of the optimization problem (see, for example, \cite{Carmona2014MasterEqnlargPoplEquilbrm, CardaliaguetetelMFG2015}).
The representation relationship between forward-backward stochastic differential equations (FBSDEs) and BSPDEs yields the so-called stochastic Feynman-Kac formula (see \cite{CBayerJQiuYao2022RoughVolBSPDEs,Hu_Ma_Yong2002SemiLinBSPDE, KaiDuQiZhang2013semiLiniearDegnBSPDEandFBSDE, JMaHYinJZhang2012non-markovianFBSDE}).
In addition, as the obstacle problems of BSPDEs, the reflected BSPDE arises as the SHJB equation for the optimal stopping problems (see \cite{qiu2014quasi}).

The study of BSPDEs dates back to more than forty years ago. The linear, semilinear, and even quasilinear BSPDEs
have been extensively studied; for Sobolev space-valued solutions in whole spaces or domains, we refer to \cite{DuQiuTang2012LpTheoryBSPDEinWholeSpace,Hu_Ma_Yong2002SemiLinBSPDE, KaiDuQiZhang2013semiLiniearDegnBSPDEandFBSDE,JMaHYinJZhang2012non-markovianFBSDE,qiu2012maximum,ZhouX1993OptimalControlforSPDEs} among many others, and for solutions 
in H{\"o}lder spaces, see \cite{CardaliaguetetelMFG2015,tang2016cauchy}. 
For the fully nonlinear cases, one may refer to \cite{JQiu2017WeakSolnNonlinSHJBEqn} for weak solution in Sobolev spaces for a special class of fully nonlinear
stochastic HJB equations, to \cite{ekren2016viscosity} for the notion of viscosity solution for path-dependent PDEs and thus for SHJB equations when the $\omega$-dependence (randomness) is through continuous path-dependence, and to \cite{JQiu2018ViscosSolSHJBEqn} for the notion of viscosity solution 
for general fully nonlinear SHJB equations as first introduced in \cite{SPeng1992SHJB}.

In many applications, numerical approximations of the SHJB equations are indispensable, as closed-form analytical solutions are generally unavailable. The numerical approximation of BSPDEs, including SHJB equations, is substantially more challenging than that of deterministic PDEs with terminal conditions or forward SPDEs for two principal reasons. First, the integrand of the stochastic integral  ($\psi$ in \eqref{SHJB-main})  is endogenous and must be determined as part of the solution. Second, the backward temporal evolution introduces conditional expectations at each time step, thereby significantly increasing both the analytical complexity and the computational cost of numerical implementations. 
By contrast, robust numerical methods based on temporal and spatial discretization have been extensively developed for deterministic PDEs and forward SPDEs; see, for example, \cite{AJentzen_PEKloeden_2009_NumerApproxSPDEs,PaoLiuChow2014SPDEs} and the references therein for comprehensive surveys of the underlying theory and numerical methodologies. Additionally, we refer to \cite{cheridito2007second,fahim2011probabilistic} and \cite{tan2014discrete} for numerical probabilistic methods for fully nonlinear parabolic PDEs without and with path-dependence respectively.

In a pioneering contribution, Dunst and Prohl \cite{Dunst2016FBStocHeatEqn} investigated a special class of linear coupled forward-backward SPDEs and established the convergence results for a finite element spatial discretization. 
For the temporal discretization, they employed the implicit Euler scheme, 
while the resulting conditional expectations were approximated using the least-squares Monte Carlo method 
in conjunction with either Picard iterations or stochastic gradient algorithms. 
In recent years, deep learning-based methods have emerged as an effective tool for the numerical approximation of nonlinear PDEs, particularly the deterministic HJB equations arising from Markovian stochastic controls. These approaches exploit the representation of the underlying PDEs via backward stochastic differential equations (BSDEs) (see, for instance, \cite{CBeckWEAJentzen2019MLforNonlinPDEAndBSDE, WEMHutAJentzTKurse2019MultiPicardNonlinPDEsandBSDE, HurePhamWarin2020DeepBSchemePDE}). Motivated by their empirical success in mitigating the curse of dimensionality and their flexibility in handling nonlinear problems, these Deep-BSDE-type algorithms have more recently been extended to the numerical approximation of BSPDEs; see \cite{CBayerJQiuYao2022RoughVolBSPDEs,HUMollaJQiu20221NumercFBSDE}. 
Another widely adopted approach for the numerical approximation of forward SPDEs is the splitting-up method, which decomposes the original problem into two simpler subproblems, typically consisting of a deterministic PDE and an SDE; see, for example, \cite{ABensoussanGlowinskiRascanu_1990_ApproxZakaiSplittinUp,ABensoussanGlowinskiRascanu_1992_ApproxSPDESplittinUp,CBeckSBeckerPCherAJentANeu2025DeepLearnAlgoSPDE}. This splitting strategy was subsequently extended to BSPDEs by Li and Tang \cite{YLiSTang_2021_ApproxBSPDESplittinUp}, where the BSPDE is decomposed into a backward deterministic PDE and a BSDE. For a class of linear BSPDEs, the authors established a first-order convergence rate for the resulting scheme.

Linearization through iterative procedures is a classical approach for the analysis and numerical solution of nonlinear equations. By reformulating a nonlinear problem as a sequence of linear equations, such methods often yield significant analytical and computational advantages. 
In stochastic control, policy (or control) iteration methods have been extensively studied for Markovian problems under various settings; see, for example \cite{Kerimkulov2020HowardPolicyImprovement,FahimRahman2025DeepPolicyControl} and the references therein. 
In particular, Kerimkulov et al. \cite{Kerimkulov2020HowardPolicyImprovement} employ a Howard-type policy improvement scheme to linearize the semilinear HJB equation, resulting in an iterative solution procedure based on a sequence of linear PDEs. In contrast,  Fahim and Rahman \cite{FahimRahman2025DeepPolicyControl} develop a policy-gradient approach that directly addresses the continuous-time stochastic control problem without explicitly solving the associated HJB equation.

To date, policy-iteration methods have been confined 
primarily to Markovian stochastic control problems, 
with no corresponding framework available for SHJB equations arising 
in the non-Markovian setting. 
In this work, we develop a policy-iteration scheme for the SHJB equation \eqref{SHJB-main}, 
in which the original nonlinear equation is successively 
approximated by a sequence of linear ones. 
We show that the resulting approximation sequence $(v^n(t,x))$ 
converges to the value function $v(t,x)$ in the mean-square sense 
for all $(t,x)$, with an exponential rate $q^n$, where $q\in(0,1/2)$. 
Furthermore, we establish a monotone improvement property of the iterates 
along the controlled state process, providing a rigorous theoretical foundation 
for the policy-improvement interpretation of the proposed scheme.

Beyond its theoretical appeal, the proposed linearization offers significant computational 
advantages. Each iteration requires only the resolution of a linear BSPDE, 
enabling the direct application of a broad range of existing numerical methods, 
including BSDE-based Monte Carlo regression, operator-splitting techniques, 
and recent deep learning approaches. 
Consequently, the proposed framework serves as a natural interface between 
policy iteration and established numerical solvers for linear BSPDEs. 
By replacing the original nonlinear SHJB equation with a sequence of linear subproblems, 
it substantially simplifies the computational complexity of each iteration while 
retaining rigorous convergence guarantees. To the best of our knowledge, 
this is the first work to establish a policy-iteration framework for SHJB equations 
together with rigorous convergence analysis 
and a computational strategy that is compatible with 
existing numerical methods for linear BSPDEs.

The rest of the paper is organized as follows. 
Section \ref{sec:Preliminaries} introduces the notations, 
states the standing assumptions, and recalls several auxiliary results. 
In Section \ref{sec:Theoretical_results_and_proofs}, we derive 
the BSDE representation for the semilinear SHJB equation 
and its linearized counterpart, and establish the convergence of the proposed policy-iteration algorithm. 
Finally, Section \ref{sec:conclusion} concludes the paper.

%%%%%%%%%%%%%%%%%%%%%%%%%%%
	
\section{Preliminaries}\label{sec:Preliminaries}
\subsection{Notations}

Let $\bN$ denote the set of strictly positive integers, set
$\bN_0:=\bN\cup \{0\}$, and write $|\cdot|$ and $\cdot$ for the Euclidean norm and inner product respectively. 
For a nonempty domain $D\subset \bR^d$, $d,l\in \bN$, and $p\in[1,\infty]$, 
we use $A^{\cT}$ to represent the transpose of a matrix $A\in \bR^{d\times l}$ and denote by $L^p(D;\bR^l)$ the space of $\bR^l$-valued $p$-th Lebesgue-measurable functions. 
In particular, the inner product and norm on $L^2(D;\bR^l)$ are given by
	$$
	\langle \phi,\,\psi\rangle=\int_{D}\phi(x)\cdot \psi(x)\,dx,\qquad \|\phi\|=\langle\phi,\,\phi\rangle^{1/2},
	\quad \text{for }\phi,\psi\in L^2(D;\bR^l).$$

	\noindent For $n\in\bN_0$ and $p\in \left[1,\infty\right]$, we denote by $(H^{n,p}(D;\bR^l),\|\cdot\|_{n,p})$ the Sobolev space
	$$H^{n,p}(D;\bR^l)=\Big\{u\in L^p(D;\bR^l):D^{\alpha}u\in L^p(D;\bR^l),\ |\alpha|\le n\Big\},$$
    where, for a multi-index $\alpha=(\alpha_1,\cdots,\alpha_d)\in \bN_0^d$, $|\alpha|=\sum_{i=1}^d\alpha_i$ and $D^{\alpha}u=\frac{\partial^{|\alpha|}u}{\partial x_1^{\alpha_1}\cdots \partial x_d^{\alpha_d}}$ denotes the corresponding weak derivative. 
    By convention, for sufficiently differentiable functions $\phi:\bR^d\rightarrow \bR^l$ and $\psi:\bR^d \rightarrow \bR$, 
     we use $D_x^{\cT}\phi$, valued in $\bR^{l\times d}$ and $D^2_{x}\psi$, valued in $\bR^{d\times d}$, to denote the Jacobian and Hessian matrices of $\phi$ with respect to the spatial variable $x$, respectively; obviously, $D_x\phi$, the transpose of Jacobian matrix, is valued in $\bR^{d\times l}$.
% , and let $(\Omega,\cF, \bP)$ be a probability space. 

Let $V$ be a Banach space with norm $\|\cdot\|_V$. For $t\in[0,T]$, let $L^2_{\cF_t}(\bP,V)$ denote the space of all $\cF_t$-measurable $V$-valued random variables $X$ satisfying
$$\|X\|_{L^2_{\cF_t}(\bP,V)}:=\Big(\bE^{\bP}\big[\|X\|_V^2\big]\Big)^{1/2}< \infty.$$
The corresponding $L^{\infty}$ space is denoted by $L^\infty_{\cF_t}(\bP,V)$ and is equipped with the essential supremum norm
\[
\|X\|_{L^\infty_{\cF_t}(\bP,V)}
:=
\operatorname*{ess\,sup}_{\omega\in\Omega}\|X(\omega)\|_V.
\]
For $\lambda\geq 0$, let $\cL^{2}_{\cP,\lambda}(\bP,V,[0,T])$ denote the space of all $\cP$-measurable $V$-valued stochastic processes $\{X_{t}\}_{t\in [0,T]}$ satisfying
\begin{equation*}
	\|X\|_{\cL_{\cP,\lambda}^2(\bP,V,[0,T])}:=\bigg(\bE^{\bP}\left[\int_{0}^{T}e^{\lambda t}\|X_t\|^2_V d t\right]\bigg)^{\frac{1}{2}}<\infty.
\end{equation*}
Similarly, we write $\cL^{\infty}_{\cP,\lambda}(\bP,V,[0,T])$ for the corresponding space endowed with the norm
\begin{equation*}
	\|X\|_{\cL^{\infty}_{\cP,\lambda}(\bP,V,[0,T])}
	:=
	\operatorname*{ess\,sup}_{(\omega,t)\in\Omega\times[0,T]}
	e^{\lambda t}\|X_t(\omega)\|_V, 
\end{equation*}
where the essential supremum is taken with respect to the product measure $\bP\otimes dt$. 
When $\lambda=0$, we use the simplified notation $(\cL_\cP^2(\bP,V),\|\cdot\|_{\cL_\cP^2(\bP,V)})$. 
Moreover, for every $\lambda\geq 0$, the norms $\|\cdot\|_{\cL_\cP^2(\bP,V)}$ and $\|\cdot\|_{\cL^2_{\cP,\lambda}(\bP,V)}$ are equivalent.

Finally, we denote by $\cS_\cP^2(\bP,V,[0,T])$ the space of all continuous, $\cP$-measurable, $V$-valued stochastic processes $\{X_{t}\}_{t\in [0,T]}$ satisfying
\begin{equation*}
	\|X\|_{\cS_\cP^2(\bP, V,[0,T])}:=\Bigg(\bE^{\bP}\bigg[\sup_{t\in [0,T]}\|X_t\|^2_V\bigg]\Bigg)^{1/2}<\infty.
\end{equation*}
It is standard that $\left(\cS_\cP^2(\bP, V),\,\|\cdot\|_{\cS_\cP^2(\bP, V)}\right)$ and $\left(\mathcal{L}_{\cP,\lambda}^2(\bP, V),\|\cdot\|_{\mathcal{L}_{\cP,\lambda}^2(\bP, V)}\right)$
are Banach spaces.

To simplify notation, arguments are occasionally omitted when no ambiguity can arise. For instance, we write $H^{n,p}$ or $H^{n,p}(\bR^l)$ in place of $H^{n,p}(D;\bR^l)$ whenever the underlying domain and dimension are clear from the context. %Also, we will use $\cL^p(\cdot)$ instead of $\cL^p_0(\cdot)$.

%%%%%%%%%%%%%%%
%%%%%%%%%%%%%%%%%%%%%%%%%%%%%%%%%%%%%%%%%%%%%%%%%%%%%%%%%%%%%%%%%%%%%%%%%%%%%%%%%%%%

\subsection{Assumptions and Auxiliary Results}
We begin by stating the standing assumptions on the coefficients and recalling several auxiliary results that will be used throughout the paper. These assumptions ensure the well-posedness of the controlled state equation \eqref{SDE-controlledstate}, the stochastic control problem \eqref{control_problem} and the associated SHJB equation \eqref{SHJB-main}. They also provide the regularity required for the convergence analysis of the proposed policy iteration scheme.

% We also introduce the associated feedback control map and summarize several technical tools that will be used in the subsequent analysis.\\

\subsubsection{Well-posedness of the Controlled SDE}
% First, we introduce the measurability assumptions on the coefficients, followed by the Sobolev regularity assumptions. 

 The following assumptions ensure the well-posedness of the controlled state equation \eqref{SDE-controlledstate}.

\begin{ass}\label{ass:meas_Lipschitz_LinearGrowth}
     \begin{enumerate} 
      \item (\textit{Measurability})  The coefficients $b, f$ are $\cP^W\otimes \cB(\bR^d) \otimes \cB(A)$-measurable and take values in $\bR^d$ and $\bR$ respectively. The coefficients $\sigma,\widetilde{\sigma}$ are $\cP^W\otimes \cB(\bR^d)$-measurable with values in $\bR^{d \times m}$ and the terminal function $G(x)\in \bR$ is $\cF^W_T$-measurable for each $x\in \bR^d$. 
        
         \item (\textit{Lipschitz continuity and Linear growth}) There exist positive constants $K$ and $\theta$ such that for all $a,\hat a\in A$, $t\in[0,T]$; and for any $x,y\in \bR^d$,
 	\begin{align*}
		&|b^{a}(t,x)-b^{\hat a}(t,y)|+|\sigma(t,x)-\sigma(t,y)|+|\widetilde{\sigma}(t,x)-\widetilde{\sigma}(t,y)|+|f^{a}(t,x)-f^{\hat a}(t,y)| 
        \\
        &\le K|x-y| + \sqrt{\theta}|a-\hat a|,
	\end{align*}
 and 
 \begin{equation*}
      |\sigma(t,x)|+|\widetilde{\sigma}(t,x)|\le K (1+|x|),\qquad |b^{a}(t,x)|+|f^{a}(t,x)|\le K(1+|x|+|a|)
 \end{equation*}
 $\bP$-almost surely. Moreover, for any $x,y\in\bR^d$,  
  \begin{equation*}
      |G(x)-G(y)|\leq K|x-y| \quad \bP\text{-a.s.}
 \end{equation*}
  \end{enumerate} 
 \end{ass}

\begin{lem}[Existence and uniqueness of strong solution to SDE]\label{wellp:FSDE}
	Assume that the coefficients $b,\sigma$ 
    and $\widetilde{\sigma}$  of the SDE \eqref{SDE-controlledstate} satisfy Assumption \ref{ass:meas_Lipschitz_LinearGrowth}. Further assume that $x$ is $\cF_t$-measurable random variable with $\bE^\bP[|x|^p]<\infty$ for some $p>1$ such that $X_t=x$. 
    Then for any $A$-valued $\bF$-progressively measurable control process $\alpha$ whose associated feedback map $\widetilde{\boldsymbol{\alpha}}(s,\cdot)$ 
    defined in \eqref{general_feedback_control_map} is uniformly Lipschitz continuous and of linear growth for every $s\in[t,T]$, 
    the controlled SDE \eqref{SDE-controlledstate} admits a unique strong solution $X^{t,x;\alpha}$. Moreover, there exists a constant $C$ depending on $p,T$ and $K$, such that 
	\begin{equation}
			\bE^\bP \Bigg[\sup_{s\in[t,T]}\Big|X_s\Big|^p \Bigg]< C\Big(1+\bE^\bP[|x|^p]\Big).
	\end{equation}
 \end{lem}

The above result is standard; see, for example, \cite[Theorem 1.3.5]{HPham2009ContstimeStoControlAndOptimizationFinancialAppl}.

\subsubsection{Well-posedness of the Control Problem and the SHJB Equation}
 We now strengthen the assumptions on the coefficients to ensure the well-posedness of the stochastic control problem \eqref{control_problem} and the associated SHJB equation \eqref{SHJB-main}.
 
\begin{ass}\label{ass:continuity_bounded_integrability}
   \begin{enumerate} 
\item (\textit{Continuity and Boundedness}) \begin{enumerate}[i)]
    \item The coefficient functions $b^{a}(t,x), f^{a}(t,x)$ and their first order derivatives in $x$ are continuous in $(t,x,a)$ and bounded by constant $K>0$ on $ [0,T]\times\bR^d\times A$, $\bP$-almost surely.
    \item The coefficient functions $\sigma(t,x), \widetilde{\sigma}(t,x)$ and their up to second-order derivatives (in $x$) are continuous in $(t,x)$ on $ [0,T]\times\bR^d$ $\bP$-almost surely.  
    Moreover, there exists a constant $K>0$ such that,
\begin{equation}
       \|\sigma(\cdot)\|_{\cL_{\cP^W}^{\infty}(\bP, H^{2,\infty})}+\|\widetilde{\sigma}(\cdot)\|_{\cL_{\cP^W}^{\infty}(\bP,H^{2,\infty})}\leq K.
    \end{equation}
    
    \item The terminal utility function $G(x)$ and its first order derivative in $x$ are continuous on $\bR^d$ $\bP$-almost surely.  Moreover, there exists a constant $K>0$ such that,
\begin{equation}
        \|G(\cdot)\|_{L_{\cF^W_T}^{\infty}(\bP, H^{1,\infty})}\leq K.
    \end{equation}
\end{enumerate}
\item (\textit{Integrability}) 
The terminal condition and running utility function further satisfy
 $$G\in L^2_{\cF_T^W}(\bP, H^{1,2})\quad\text{and}\quad\sup_{a\in A}f^a\in \cL^2_{\cP^W}(\bP, H^{1,2}).$$
    \end{enumerate}
\end{ass}

% Assumption \ref{ass:meas_Lipschitz_LinearGrowth} guarantees the measurability, Lipschitz continuity, and linear growth conditions required for the well-posedness of the controlled state equation \eqref{SDE-controlledstate}. However, these conditions alone do not provide the Sobolev regularity needed for the BSPDE framework, particularly on the unbounded domain $\bR^d$. Assumption \ref{ass:continuity_bounded_integrability} therefore imposes additional continuity, boundedness, and integrability requirements on the coefficients and their spatial derivatives in appropriate Sobolev spaces.

% It is worth noting that the uniform boundedness of the coefficients and their first-order derivatives in Assumption \ref{ass:continuity_bounded_integrability} implies the Lipschitz continuity and linear growth conditions in Assumption \ref{ass:meas_Lipschitz_LinearGrowth}. Nevertheless, the latter assumption is stated separately to emphasize the minimal conditions required for the well-posedness of the controlled state equation, while the additional Sobolev regularity conditions are needed for the analysis of the SHJB equation.\\

Under Assumptions \ref{ass:meas_Lipschitz_LinearGrowth} and \ref{ass:continuity_bounded_integrability}, the stochastic control problem \eqref{control_problem} is well defined, and the value function \eqref{value_fnc} is finite. To study the well-posedness of the SHJB equation \eqref{SHJB-main}, we introduce the Hamiltonian
 $$\widetilde{H}(t,x,p):=\sup_{a\in A}\big(b^a\cdot p+f^a\big)(t,x),\qquad (t,x,p)\in[0,T]\times\bR^d\times\bR^d.$$
 
 By Assumption \ref{ass:continuity_bounded_integrability}, $\widetilde{H}(\cdot,\cdot,p)$ is $\cP^W \otimes \cB(\bR^d)$-measurable for every $p\in\bR^d$. The uniform boundedness of $b^a$ implies that $\widetilde{H}$ is uniformly Lipschitz continuous in $p$; in particular, whenever the derivative $\widetilde{H}_p$ exists, it is uniformly bounded by $K$. Moreover, whenever the spatial derivative $\widetilde{H}_x$ exists, the boundedness of $D_xb^a$ and $D_xf^a$ yields the linear growth estimate $|\widetilde{H}_x(t,x,p)|\le K(1+|p|)$, up to a change of the constant $K$. Finally, the integrability condition on $f$ implies that $\widetilde{H}(\cdot,\cdot,0)\in \cL_{\cP^W}^2(\bP,H^{1,2})$. These properties provide the regularity required by the BSPDE well-posedness theory used below. \\

We further impose the following super-parabolicity (equivalently, uniform non-degeneracy) condition on the diffusion coefficient $\widetilde{\sigma}$, which is required for the BSPDE well-posedness theory underlying the SHJB equation \eqref{SHJB-main}.

\begin{ass}\label{ass:nondegeneracy/parabolicity}
	There exists a constant $\gamma>0$ such that for any vector $\xi\in\bR^d$ and $(t,x)\in[0,T]\times\bR^d$,
    \begin{equation}
        \xi^T \widetilde{\sigma}(t,x)\widetilde{\sigma}(t,x)^{\cT}\xi\geq \gamma |\xi|^2, \quad \text{a.s.}
    \end{equation}
\end{ass}

Under Assumption \ref{ass:nondegeneracy/parabolicity}, the diffusion matrix $\widetilde{\sigma}(t,x)$ is uniformly non-degenerate, which in particular yields full row rank (hence $m\ge d$). Consequently, the Moore--Penrose pseudoinverse
\[
\widetilde{\sigma}^{+}:=\widetilde{\sigma}^{\cT}\big(\widetilde{\sigma}\widetilde{\sigma}^{\cT}\big)^{-1}
\]
is well defined. Furthermore, it obeys the uniform estimate
\[
\sup_{(t,x)\in[0,T]\times\bR^d}\big|\widetilde{\sigma}^{+}(t,x)\big|\le \gamma^{-1/2},\quad \text{a.s.}
\]

\begin{defn}\label{defn:strong_solution}
A pair of stochastic processes $(v,\psi)\in \big(\cL_{\cP^W}^2(\bP, H^{2,2})\cap \cS_{\cP^W}^2(\bP, H^{1,2})\big)\times \cL_{\cP^W}^2(\bP, H^{1,2})$ is called a (strong) solution of the SHJB equation \eqref{SHJB-main} 
if for almost every $(\omega,x)\in \Omega\times\bR^d$, it holds that
\begin{align*}
    v(t,x)
    &=G(x)+\int_t^T \Big(\widetilde{H}(s,x,D_xv(s,x))+ tr\Big[\frac{1}{2}\big(\sigma\sigma^T+\widetilde{\sigma}\widetilde{\sigma}^T\big)D^2_x v+ D_x\psi\sigma^{\cT} \Big](s,x)\Big)\, ds
    \\
    &\quad -\int_t^T \psi(s,x)dW_s, \quad \forall\, t\in[0,T].
\end{align*}
\end{defn}

\begin{lem}[Existence and uniqueness of the solution of SHJB]\label{wellpossed_SHJB}
	Suppose that Assumptions \ref{ass:meas_Lipschitz_LinearGrowth}, \ref{ass:continuity_bounded_integrability} and \ref{ass:nondegeneracy/parabolicity} hold. 
    Then the SHJB equation \eqref{SHJB-main} admits a unique solution $\big(v,\psi\big)\in \big(\cL_{\cP^W}^2(\bP, H^{2,2})\cap \cS_{\cP^W}^2(\bP, H^{1,2})\big)\times \cL_{\cP^W}^2(\bP, H^{1,2})$ satisfying
    \begin{equation}
        \|v\|^2_{\cL^2_{\cP^W}( H^{2,2})}+\|v\|^2_{\cS^2_{\cP^W}( H^{1,2})}+\|\psi\|^2_{\cL^2_{\cP^W}(H^{1,2})}\leq C\Big(\big\|\sup_{a\in A }f^a\big\|^2_{\cL^2_{\cP^W}(H^{1,2})}+\|G\|^2_{L^2_{\cF_T^W}( H^{1,2})}\Big),
    \end{equation}
    where the constant $C$ depends on $d,K,\gamma$, and $T$.
\end{lem}
The above result is standard; see, for example, \cite[Theorem 3.1]{KaiDuQiZhang2013semiLiniearDegnBSPDEandFBSDE} or the general $L^p$-theory \cite{DuQiuTang2012LpTheoryBSPDEinWholeSpace}. 
Now, the following verification theorem provides the link between the SHJB equation \eqref{SHJB-main} and the stochastic optimal control problem \eqref{control_problem}.

\begin{lem}[Verification Theorem]\label{verification_theorem}
Suppose that the Assumptions of Lemma \ref{wellp:FSDE} hold. Let $\big(v,\psi\big)\in \big(\cL_{\cP^W}^2(\bP, H^{2,2})\cap \cS_{\cP^W}^2(\bP, H^{1,2})\big)\times \cL_{\cP^W}^2(\bP, H^{1,2})$ denote the unique solution of the SHJB equation \eqref{SHJB-main}. Then, for any given $x\in L^2_{\cF_t}(\bP, \bR^d)$, $t\in[0,T]$ and $\alpha\in\cA$ we have $v(t,x)\geq J(t,x,\alpha)$ $\bP$-almost surely. Furthermore, assume that the function $H$ in \eqref{hamiltonian} attains its supremum at $\alpha^*_s\in A$ $\bP$-almost surely for any $s\in[t,T]$ such that $\alpha^*:=(\alpha^*_s)_{s\in[t,T]}\in \cA$. Then, we have for any given $x\in L^2_{\cF_t}(\bP, \bR^d)$, $$v(t,x)=J(t,x;\alpha^*)=\esssup_{\alpha \in \cA} J(t,x;\alpha)$$ $\bP$-almost surely for every $t\in[0,T]$,  and $\alpha^*$ is the corresponding optimal control. 
\end{lem}
Through standard verification arguments (cf. \cite{SPeng1992SHJB}), the proof follows straightforwardly from the generalized It\^o-Kunita-Wentzell formula below. 
\begin{lem}\label{Ito_kunita_formula} (Generalized It\^o-Kunita-Wentzell formula by Yang and Tang \cite[Theorem 3.1]{YangTang2013SdeRandomCoeff_BSPDE}) Suppose that the Assumptions \ref{ass:meas_Lipschitz_LinearGrowth}, \ref{ass:continuity_bounded_integrability} and \ref{ass:nondegeneracy/parabolicity} hold. Let $X^{t,x;\alpha}$ denote the unique strong solution to the controlled SDE \eqref{SDE-controlledstate} governed by the control process $\alpha\in\cA$. 
     Also suppose that the random field $u:\Omega\times [0,T]\times \bR^d \rightarrow \bR$ satisfies, 
     for every $\tau \in[0,T]$,
	\begin{equation}
		u(\tau,x)=u(0,x)+\int_0^{\tau }F(t,x)dt + \int_0^{\tau}\psi (t,x)dW_t,\quad a.e.\;x\in\bR^d,
	\end{equation}
$\bP$-almost surely, where $u\in \cL_{\cP^W}^2(H^{2,2})\cap\cS_{\cP^W}^2(H^{1,2}) ,\; \psi \in \cL_{\cP^W}^2(H^{1,2})$ and $F\in\cL_{\cP^W}^2(L^2)$. Then, for each $\tau \in[t,T]$, $\alpha\in\cA$, and almost every $x\in\bR^d$, we have 
\begin{align}
	u\left(\tau,X_{\tau}^{t,x;\alpha}\right) &= u(t,x) + \int_t^{\tau}\Big( tr\Big[\frac{1}{2}\big(\sigma\sigma^\cT+\widetilde{\sigma}\widetilde{\sigma}^\cT\big)D^2_x u+ D_x\psi\sigma^{\cT} \Big]+(b^{\alpha})^{\cT}D_xu + F\Big)\left(s,X_s^{t,x;\alpha}\right) ds\nonumber \\
	& +\int_t^{\tau}(\psi+\sigma^{\cT} D_x u)\left(s,X_s^{t,x;\alpha}\right) dW_s 
    +\int_t^{\tau} (\widetilde{\sigma}^\cT D_x u )\left(s,X_s^{t,x;\alpha}\right)d\widetilde{W}_s,\quad \text{a.s.}
\end{align}
\end{lem}

\subsection{Feedback Control Map and Change of Measure}

We first introduce a feedback selector associated with the Hamiltonian maximization problem.
\begin{defn}
    For $(t,x,\tilde{z})\in [0,T]\times \bR^d \times \bR^m$, a feedback control map $\boldsymbol{\alpha}$ is a $\cP^W\otimes
\cB(\mathbb R^d)\otimes
\cB(\mathbb R^m)$-measurable function satisfying
		\begin{equation}\label{def:feedback_control_a(t,x,z)}
			\boldsymbol{\alpha}(t,x,\tilde{z})\in \arg\max_{a\in A} \Big( \big(\widetilde{\sigma}^{+}b^{a}\big)^\cT(t,x)\tilde{z}+f^{a}(t,x)\Big).
		\end{equation}
\end{defn}
Under Assumptions \ref{ass:meas_Lipschitz_LinearGrowth} and \ref{ass:continuity_bounded_integrability}, the mapping
\[
(t,x,\tilde{z},a)\mapsto \big(\widetilde \sigma^{+}b^{a} \big)^\cT(t,x)\tilde{z}+f^{a}(t,x)
\]
is $\cP^W\otimes\cB(\mathbb R^d)\otimes
\cB(\mathbb R^m)\otimes\cB(A)$-measurable and continuous in $a$. Since the control set $A$ is compact, the maximum is attained, 
and the measurable selection theorem yields a measurable maximizer. Consequently, the feedback map $\boldsymbol{\alpha}$ in \eqref{def:feedback_control_a(t,x,z)} may be chosen to be 
$\cP^W\otimes
\cB(\mathbb R^d)\otimes
\cB(\mathbb R^m)$-measurable.
The convergence analysis requires the following regularity of the feedback control map $\boldsymbol{\alpha}(t,x,\tilde{z})$.

\begin{ass}\label{ass:control_lipschitz_bound}
		There exist constants $K,\theta >0$ such that 
	 for all $x,x'\in\bR^d$, $\tilde{z},\tilde{z}'\in\bR^m$, and $t\in[0,T]$,
	\begin{equation*}
	    |\boldsymbol{\alpha}(t,x,\tilde{z})-\boldsymbol{\alpha}(t,x,\tilde{z}')|\leq \sqrt{\theta}|\tilde{z}-\tilde{z}'|, \quad \text{a.s.,}
	\end{equation*}
and
\begin{equation*}
    |\boldsymbol{\alpha}(t,x,\tilde{z})-\boldsymbol{\alpha}(t,x',\tilde{z})|\leq K|x-x'|,\quad \text{a.s.}
\end{equation*}
\end{ass}

\begin{rmk}
Assumption \ref{ass:control_lipschitz_bound} is satisfied in natural classes of examples. 
We give one such example on the whole space. Assume, for simplicity, that \(m=d\) and
\[
    \widetilde\sigma(t,x)=I_d.
\]
Let \(A=\overline{B_R(0)}\subset\mathbb R^l\), let \(B\in\mathbb R^{d\times l}\), and suppose
\[
    b^a(t,x)=b_0(t,x)+\ell(x)g(x)Ba,
\]
and
\[
    f^a(t,x)=f_0(t,x)-\frac{\rho}{2}\ell(x)|a|^2,
    \qquad \rho>0.
\]
Here \(b_0\) and \(f_0\) are assumed to satisfy the standing measurability, boundedness, and Sobolev regularity assumptions. Assume also that
\[
    \ell\in C_b^1(\mathbb R^d)\cap H^{1,2}(\mathbb R^d),
    \qquad
    \ell(x)\ge0,
\]
and
\[
    g\in C_b^1(\mathbb R^d),
    \qquad
    0<g_*\le g(x)\le g^*<\infty.
\]
Then the \(a\)-dependent part of the Hamiltonian is
\[
    \ell(x)\left(g(x)(Ba)^\cT\tilde z-\frac{\rho}{2}|a|^2\right).
\]
For every \(x\) such that \(\ell(x)>0\), maximizing the Hamiltonian over \(a\in A\) is equivalent to maximizing
\[
    g(x)(Ba)^\cT\tilde z-\frac{\rho}{2}|a|^2.
\]
Therefore,
\[
    \boldsymbol{\alpha}(t,x,\tilde z)
    =
    \Pi_A\left(\frac{g(x)}{\rho}B^\cT\tilde z\right),
\]
where \(\Pi_A\) denotes the Euclidean projection onto \(A\). At points where \(\ell(x)=0\), the Hamiltonian is independent of \(a\), and the same formula may be used as a measurable tie-breaking rule.

Since the projection onto a closed convex set is non-expansive, we have
\[
    |\boldsymbol{\alpha}(t,x,\tilde z)-\boldsymbol{\alpha}(t,x,\tilde z')|
    \le
    \frac{g^*\|B\|}{\rho}|\tilde z-\tilde z'|.
\]
Moreover, because \(A=\overline{B_R(0)}\) and \(g\) is bounded away from zero, the map
\[
    x\mapsto 
    \Pi_A\left(\frac{g(x)}{\rho}B^\cT\tilde z\right)
\]
is Lipschitz uniformly in \(\tilde z\). Indeed,
\[
    |\boldsymbol{\alpha}(t,x,\tilde z)-\boldsymbol{\alpha}(t,x',\tilde z)|
    \le
    \frac{R}{g_*}|g(x)-g(x')|
    \le
    \frac{R\|Dg\|_\infty}{g_*}|x-x'|.
\]
Hence Assumption \ref{ass:control_lipschitz_bound} holds. The role of the factor \(\ell(x)\) is to provide spatial decay of the \(a\)-dependent coefficients on \(\mathbb R^d\), while the bounded positive factor \(g(x)\) allows the optimizer to depend on the state variable \(x\).
\end{rmk}

% \subsubsection{Regularity of the Value Function}
In the non-Markovian setting, classical differentiability of the value function $v$ and of the auxiliary process $\psi$ is generally unavailable. Nevertheless, the following regularity properties are sufficient for the subsequent analysis. 
Arguing as in \cite[Proposition 3.3]{JQiu2018ViscosSolSHJBEqn}, Assumptions \ref{ass:meas_Lipschitz_LinearGrowth} and \ref{ass:continuity_bounded_integrability} imply that, 
for every $(\alpha,x)\in\cA\times\bR^d$, the process $\{v(s,X_s^{t,x;\alpha})\}_{s\in[t,T]}$ has continuous paths. Moreover, there exists a constant $M>0$, depending only on $K$ and $T$, such that for each $t\in[0,T]$,
\begin{equation}
	    |v^0(t,x)-v^0(t,y)|+ |v(t,x)-v(t,y)|\leq M|x-y|;\quad x,y\in\bR^d, \quad \bP\text{-a.s.,}
\end{equation}
where $v^0(t,x)$ may be thought of as the gain functional $J(t,x;a^0)$ with $a^0(t,x)\equiv \bar a_0\in A$ as set in Algorithm 1. Consequently, for each fixed $t$, the map $x\mapsto v(t,x)$ is Lipschitz continuous and hence differentiable almost everywhere with respect to $x$ and whenever the spatial gradients exist,
\begin{equation}\label{value_fnc_gradient_bound}
	    |D_xv^0(t,x)|+|D_xv(t,x)|\leq M,\quad \text{$\bP$-a.s.}
\end{equation}

By the well-posedness result in Lemma \ref{wellpossed_SHJB} and the verification theorem in Lemma \ref{verification_theorem}, the SHJB equation \eqref{SHJB-main} admits a unique solution $(v,\psi)$, whose first component coincides with the value function of the stochastic control problem \eqref{control_problem}. The associated probabilistic representation, together with the feedback selector \eqref{def:feedback_control_a(t,x,z)}, motivates the following candidate optimal feedback control:

\begin{equation}\label{optimal_feedback_control_law}
	    \alpha^*_s
:=
\boldsymbol{\alpha}
\bigl(
s,
X_s,
\big(\widetilde{\sigma}^\cT D_xv\big)(s,X_s)
	\bigr).
\end{equation}

Since the feedback map $\boldsymbol{\alpha}$ takes values in the compact control set $A$, the process $(\alpha_s^*)_{s\in[0,T]}$ is bounded. Its admissibility, however, is not automatic. Although $\boldsymbol{\alpha}(s,x,\tilde{z})$ is uniformly Lipschitz continuous in $x$ and $\tilde{z}$ by Assumption \ref{ass:control_lipschitz_bound}, the closed-loop map

$$(s,x)
\longmapsto
\boldsymbol{\alpha}
\bigl(
	s,x,\big(\widetilde{\sigma}^\cT D_xv\big)(s,x)
	\bigr)$$
need not be Lipschitz continuous in $x$, since $D_xv$ is only known to be bounded and measurable. If this closed-loop map is Lipschitz continuous in the state variable, then Lemma \ref{wellp:FSDE} implies that the closed-loop state equation 
\[
dX_s
=
b^{\alpha_s^*}(s,X_s)\,ds
+\sigma(s,X_s)\,dW_s+\widetilde{\sigma}(s,X_s)\,d\widetilde{W}_s,
\quad s\in[t,T], \quad X_t=x,
\]
admits a unique strong solution for every \((t,x)\in[0,T]\times\mathbb R^d\), and the feedback process $\alpha^*$ is admissible. This condition is satisfied, for instance, if the spatial gradient $D_xv$ is Lipschitz continuous in $x$; in that case, the Lipschitz regularity of $\widetilde{\sigma}$ and Assumption \ref{ass:control_lipschitz_bound} yield the required Lipschitz continuity of the closed-loop map.

To avoid imposing or proving additional regularity of $D_xv$, and to focus on the convergence analysis, we make the following admissibility assumption.
\begin{ass}\label{ass:feedback_control_lipschitz}
	   The candidate feedback control process defined in \eqref{optimal_feedback_control_law} belongs to the admissible class $\cA$.
\end{ass}
Under this assumption, Lemma \ref{verification_theorem} implies that $\alpha^*$ is an optimal control for the stochastic control problem \eqref{control_problem}.

We next introduce an equivalent probability measure that will be used in the subsequent analysis. 
The construction relies on the Cameron--Martin--Girsanov theorem. The boundedness of $b^a$, together with the super-parabolicity condition in Assumption \ref{ass:nondegeneracy/parabolicity}, 
implies that $\widetilde{\sigma}^{+}b^a$ is uniformly bounded. 
Consequently, Novikov's condition is satisfied for the feedback control \eqref{optimal_feedback_control_law}, and the following change of measure is well defined.
% \begin{lem}[Cameron-Martin-Girsanov theorem]
% 	Let $\gamma=\{\gamma_t: t\in [0,T]\}$ and $\widetilde{\gamma}=\{\widetilde{\gamma}_t: t\in [0,T]\}$ be two $\cP$-measurable processes such that 
% 	\begin{equation}
% 		\bE^{\bP} \Bigg[exp \Bigg(\frac{1}{2} \int_0^T \gamma^2_t dt\Bigg)\Bigg]< \infty\quad \text{and}\quad \bE^{\bP} \Bigg[exp \Bigg(\frac{1}{2} \int_0^T \widetilde{\gamma}^2_t dt\Bigg)\Bigg]< \infty.
% 	\end{equation}
% There exists a measure $\bQ$ on $\big(\Omega, \cF\big)$ such that
% \begin{enumerate}
% 	\item $\bQ$ is equivalent to $\bP$.
% 	\item $\frac{d\bQ}{d\bP}=exp\Big[-\int_0^T \gamma_t dW_t -\int_0^T \widetilde{\gamma}_t d\widetilde{W}_t-\frac{1}{2} \int_0^T \gamma^2_t dt-\frac{1}{2} \int_0^T \widetilde{\gamma}^2_t dt\Big]$.
% 	\item The processes $B=\Big\{ B_t : t\in [0,T]\Big\}$ and  $\widetilde{B}=\Big\{ \widetilde{B}_t : t\in [0,T]\Big\}$ defined as $B_t=W_t+\int_0^t \gamma_s ds$  and $\widetilde{B}_t=\widetilde{W}_t+\int_0^t \widetilde{\gamma}_s ds$ are two $\big(\bF,\bQ\big)$-Brownian motions.
	
% \end{enumerate}
% \end{lem}

\begin{prop}\label{change_measure} 
	    Let Assumptions \ref{ass:meas_Lipschitz_LinearGrowth}--\ref{ass:feedback_control_lipschitz} hold. Put
	    \[
	    \eta_t:=\big(\widetilde{\sigma}^{+}b^{\alpha^*_t}\big)(t,X_t),\qquad t\in[0,T].
	    \]
	    Define the probability measure $\bQ$ by
		\begin{equation}
			\frac{d\bQ}{d\bP}=\exp\Bigg(-\int_0^T \eta_t d\widetilde{W}_t-\frac{1}{2}\int_0^T|\eta_t|^2 dt\Bigg).
		\end{equation}
	Then, by the Cameron--Martin--Girsanov theorem, $\bQ$ is equivalent to $\bP$, and the processes $B=\{ B_t : t\in [0,T]\}$ and $\widetilde{B}=\{ \widetilde{B}_t : t\in [0,T]\}$ defined by
	\begin{equation}
		 B_t:=W_t,\qquad \widetilde{B}_t:= \widetilde{W}_t +\int_0^t \eta_s\,ds,
	\end{equation}
	are two independent $\big(\bF,\bQ\big)$-Brownian motions. 
\end{prop}

Since $\bQ$ is equivalent to $\bP$, null sets are unchanged; in particular, $\bF$-adaptedness and progressive measurability are preserved under the change of measure.
We finally note that, if the admissibility requirement were relaxed so that the controlled state equation \eqref{SDE-controlledstate} were required only to admit a weak solution, then Proposition \ref{change_measure} could in principle be used to construct such a solution and to establish the admissibility of the feedback control \eqref{optimal_feedback_control_law}. In that formulation, Assumption \ref{ass:feedback_control_lipschitz} would no longer be necessary. Developing this weak formulation, however, would require substantial additional arguments for both the non-Markovian control problem and the associated BSPDE theory; such a framework does not appear to be available in the existing literature.

%%%%%%%%%%%%%%%%%%%%%%%%%%%%
%%%%%%%%%%%%%%%%%%%%%%%%%%%%
%%%%%%%%%%%%%%%%%%%%%%%%%%

\section{Convergence Analysis}\label{sec:Theoretical_results_and_proofs}
Throughout this section, \(X\equiv X^{t,x}\) denotes the solution of the controlled SDE \eqref{SDE-controlledstate} with initial condition \(X_t=x\in L^2_{\cF_t}(\bP,\bR^d)\), 
driven by the optimal feedback control \(\alpha^*\in\cA\) defined in \eqref{optimal_feedback_control_law}. 
Thus,
\begin{equation}\label{SDE-optimalcontrolled}
	\begin{cases}
		dX_s &=b^{\alpha_s^*} (s,X_s)ds + \sigma(s,X_s)dW_s + \widetilde{\sigma}(s,X_s)d\widetilde{W}_s, \quad s\in[t,T],\\
		X_t &=x.
	\end{cases}		
\end{equation}

\subsection{BSDE Representations}
The convergence analysis is based on the connection between non-Markovian backward stochastic differential equations (BSDEs) and backward stochastic partial differential equations (BSPDEs). Since the argument is carried out at the BSDE level, we first derive the BSDE representations of the BSPDEs \eqref{SHJB-main} and \eqref{SHJB-Algorithm}. These representations are given in the following two propositions.

% Stochastic Feynman-Kac (non-Markovian) equivalent
\begin{prop}[BSDE representation of the linear BSPDE]\label{prop:linearBSPDE2BSDE}
    Suppose that Assumptions \ref{ass:meas_Lipschitz_LinearGrowth}--\ref{ass:feedback_control_lipschitz} hold. Let $(v^n, \psi^n)\in\big(\cL_{\cP^W}^2(\bP, H^{2,2})\cap \cS_{\cP^W}^2(\bP, H^{1,2})\big)\times \cL_{\cP^W}^2(\bP, H^{1,2})$ be the unique solution of the linear BSPDE \eqref{SHJB-Algorithm}. For $\tau\in[t,T]$, define
    \begin{equation}\label{v^n:defn}
 	Y_\tau^n:=v^n(\tau,X_\tau),\quad
 	Z^n_\tau:=\big(\psi^n+\sigma^\cT D_xv^n\big)(\tau,X_\tau),\quad
 	\widetilde{Z}_\tau^n:=\big(\widetilde{\sigma}^\cT D_xv^n\big)(\tau,X_\tau).
    \end{equation}
  For $s\in[t,T]$ and $\tilde{z}_1,\tilde{z}_2\in\bR^m$, introduce the driver
 \begin{equation}\label{driver_BSDE}
 	F_s(\tilde{z}_1,\tilde{z}_2):= \big(\widetilde{\sigma}^{+} b^{\boldsymbol{\alpha}(s,X_s,\tilde{z}_1)}\big)^{\cT}(s,X_s)\tilde{z}_2 +f^{\boldsymbol{\alpha}(s,X_s,\tilde{z}_1)}(s,X_s),
 \end{equation}
    where $\boldsymbol{\alpha}$ is the feedback control map defined in \eqref{def:feedback_control_a(t,x,z)}. Then the triplet $(Y^n, Z^n, \widetilde{Z}^n)$ satisfies, for every $\tau\in[t,T]$, the BSDE
\begin{align} \label{BSDE_v^n_alg}
	Y^n_\tau = G(X_T) + \int_\tau^T F_s(\widetilde{Z}^{n-1}_s,\widetilde{Z}^n_s)ds - \int_\tau^T Z^n_s dB_s-\int_\tau^T\widetilde{Z}^n_s d\widetilde{B}_s,\quad\bQ\text{-a.s.,}
\end{align}
where $B$ and $\widetilde{B}$ are the Brownian motions under $\bQ$ introduced in Proposition \ref{change_measure}.
\end{prop}

\begin{proof}
    Let $(v^n, \psi^n)$ be the solution to the linear BSPDE \eqref{SHJB-Algorithm}. Recall that the corresponding feedback control is given by
	\begin{equation}
		a^n(s,x)=\arg\max_{a\in A} \Big(\big( (b^{a})^{\cT}D_xv^{n-1}\big)(s,x)+f^{a}(s,x)\Big)
		=\boldsymbol{\alpha}\Big(s,x, \big(\widetilde{\sigma}^\cT D_x v^{n-1}\big)(s,x)\Big).
	\end{equation}
By Assumptions \ref{ass:meas_Lipschitz_LinearGrowth} and \ref{ass:continuity_bounded_integrability}, the coefficients of the linear BSPDE \eqref{SHJB-Algorithm} satisfy the assumptions of Lemma \ref{Ito_kunita_formula}. 
Applying the generalized It\^o-Kunita-Wentzell formula (Lemma \ref{Ito_kunita_formula}) to $v^n(s,X_s)$, we obtain for $s\in[t,T]$,
\begin{align}\label{Ito-kunita}
	dY_s^n
	=&\,dv^n(s,X_s) \nonumber\\
	=&\, tr\Big[\frac{1}{2}(\sigma\sigma^\cT+\widetilde{\sigma}\widetilde{\sigma}^\cT)D^2_xv^n
    + D_x\psi^n\sigma^\cT\Big](s,X_s)ds 
    +((b^{\alpha_s^*})^{\cT}D_xv^n)(s,X_s)ds \nonumber\\
	&-tr\Big[\frac{1}{2}(\sigma\sigma^\cT+\widetilde{\sigma}\widetilde{\sigma}^\cT)D^2_xv^n+D_x\psi^n\sigma^\cT\Big](s,X_s)ds -((b^{a^n})^{\cT}D_xv^n)(s,X_s)ds \nonumber\\
	&-f^{a^n}(s,X_s)ds + (\psi^n+\sigma^\cT D_xv^n)(s,X_s)dW_s +\big(\widetilde{\sigma}^\cT D_xv^n\big) (s,X_s)d\widetilde{W}_s \nonumber\\
	=&\, \Big(((b^{\alpha_s^*})^{\cT}D_xv^n)(s,X_s) -((b^{a^n})^{\cT}D_xv^n)(s,X_s) -f^{a^n}(s,X_s) \Big) ds \nonumber \\
	& + (\psi^n+\sigma^\cT D_xv^n)(s,X_s)dW_s +\big(\widetilde{\sigma}^\cT D_xv^n \big)(s,X_s) d\widetilde{W}_s .
\end{align}
Recalling the definitions in \eqref{v^n:defn}, we obtain, for every $\tau\in[t,T]$,
\begin{align}\label{BSDE_vn_under_P}
	Y^n_\tau &= G(X_T) -\int_\tau^T \Big[\big(\widetilde{\sigma}^{+}b^{\alpha^*_s}\big)^\cT(s,X_s)\widetilde{Z}^n_s -F_s(\widetilde{Z}^{n-1}_s,\widetilde{Z}^n_s)\Big]ds - \int_\tau^T Z^n_s dW_s
    \notag\\
    &\quad
    -\int_\tau^T \widetilde{Z}^n_s d\widetilde{W}_s.
\end{align}
By Proposition~\ref{change_measure}, 
\[
B_t=W_t,\qquad 
\widetilde B_t=\widetilde W_t+\int_0^t\big(\widetilde{\sigma}^{+}b^{\alpha_s^*}\big)(s,X_s)\,ds
\]
are two \((\bF,\bQ)\)-Brownian motions. Consequently, \eqref{BSDE_vn_under_P} can be rewritten as
\begin{align*}
	Y^n_\tau = G(X_T) + \int_\tau^T F_s(\widetilde{Z}^{n-1}_s,\widetilde{Z}^n_s)ds - \int_\tau^T Z^n_s dB_s-\int_\tau^T\widetilde{Z}^n_s d\widetilde{B}_s.
\end{align*}
This is \eqref{BSDE_v^n_alg}, and the proof is complete.
\end{proof}

\begin{prop}\label{prop:semilinearSHJB2BSDE}
Suppose that Assumptions \ref{ass:meas_Lipschitz_LinearGrowth}--\ref{ass:feedback_control_lipschitz} hold.
Then the SHJB equation \eqref{SHJB-main} admits a unique solution 
$(w,\psi)\in\big(\cL_{\cP^W}^2(\bP, H^{2,2})\cap \cS_{\cP^W}^2(\bP, H^{1,2})\big)\times \cL_{\cP^W}^2(\bP, H^{1,2})$ with $w$ coinciding with the value function $v$ 
of the stochastic control problem \eqref{control_problem}. 
Moreover, for $s\in[t,T]$, set
\begin{equation}\label{change_of_variable_w}
	Y_s=w(s,X_s),\quad 
    \widetilde{Z}_s=\big(\widetilde{\sigma}^\cT D_x w\big)(s,X_s),
     \quad
     Z_s=\big(\psi+\sigma^\cT D_x w\big)(s,X_s). 
\end{equation}
Then the triplet $(Y,Z,\widetilde{Z})$ satisfies the BSDE
\begin{align}\label{BSDE_v_valuefnc}
	Y_{\tau}&= G(X_T) + \int_{\tau}^T F_s (\widetilde{Z}_s , \widetilde{Z}_s)ds - \int_{\tau}^T Z_s dB_s-\int_{\tau}^T \widetilde{Z}_s d\widetilde{B}_s,\quad\bQ\text{-a.s. for } \tau\in[t,T].
\end{align}
\end{prop}

\begin{proof}
By Lemma \ref{wellpossed_SHJB}, the SHJB equation \eqref{SHJB-main} admits a unique (strong) solution
\[
(w,\psi)\in
\big(\cL_{\cP^W}^2(\bP,H^{2,2})\cap \cS_{\cP^W}^2(\bP,H^{1,2})\big)
\times \cL_{\cP^W}^2(\bP,H^{1,2}).
\]
The verification theorem, Lemma \ref{verification_theorem}, then identifies $w$ with the value function $v$ of \eqref{control_problem}. It remains to derive the BSDE representation.

% For notational brevity, write
% \[
% \mathcal L w:=
% tr\Big[\frac{1}{2}\big(\sigma\sigma^T+\widetilde{\sigma}\widetilde{\sigma}^T\big)D_x^2w+\sigma D_x\psi\Big].
% \]
The solution $(w,\psi)$ satisfies, for a.e. $(\omega,x)$ and all $\tau\in[t,T]$,
\begin{align*}
     w(\tau,x)&
     =G(x)
     +\int_\tau^T \Big(
    \text{tr}\Big[\frac{1}{2}\big(\sigma\sigma^\cT+\widetilde{\sigma}\widetilde{\sigma}^\cT\big)D_x^2w+ D_x\psi\sigma^\cT\Big]
+\sup_{a\in A}((b^a)^\cT D_xw+f^a)\Big)(s,x)\,ds\\
&\quad
-\int_\tau^T\psi(s,x)\,dW_s.
\end{align*}
By Assumptions \ref{ass:meas_Lipschitz_LinearGrowth} and \ref{ass:continuity_bounded_integrability}, 
 the generalized It\^o-Kunita-Wentzell formula, Lemma \ref{Ito_kunita_formula}, applies to the random field $w$ along the optimally controlled state process $X$. Using the definitions in \eqref{change_of_variable_w}, we obtain
\begin{align}\label{BSDE-under-P}
	dY_\tau
	&=
	\Big[
		((b^{\alpha_\tau^*})^\cT D_xw)(\tau,X_\tau)
		-\sup_{a\in A}((b^a)^\cT D_xw+f^a)(\tau,X_\tau)
	\Big]\,d\tau
	+Z_\tau\,dW_\tau+\widetilde Z_\tau\,d\widetilde W_\tau .
\end{align}

Since $\widetilde Z_\tau=(\widetilde\sigma^\cT D_xw)(\tau,X_\tau)$ and 
\[
\alpha_\tau^*
=\boldsymbol{\alpha}(\tau,X_\tau,\widetilde Z_\tau),
\]
the definition of the feedback selector \eqref{def:feedback_control_a(t,x,z)} gives
\[
\sup_{a\in A}((b^a)^\cT D_xw+f^a)(\tau,X_\tau)
=(b^{\alpha_\tau^*})^\cT D_xw(\tau,X_\tau)+f^{\alpha_\tau^*}(\tau,X_\tau)
=F_\tau(\widetilde Z_\tau,\widetilde Z_\tau).
\]
Moreover, as \(\widetilde\sigma\widetilde\sigma^+=I_d\), it follows that
\[
((b^{\alpha_\tau^*})^\cT D_xw)(\tau,X_\tau)
=\big(\widetilde\sigma^{+}b^{\alpha_\tau^*}\big)^\cT (\tau,X_\tau)\widetilde Z_\tau .
\]
Thus, integrating \eqref{BSDE-under-P} from $\tau$ to $T$ and using $Y_T=w(T,X_T)=G(X_T)$, we obtain
\begin{align}\label{BSDE-value-under-P}
	Y_\tau
	&=G(X_T)
	-\int_\tau^T
	\Big[
		\big(\widetilde\sigma^{+}b^{\alpha_s^*}\big)^\cT (s,X_s)\widetilde Z_s
		-F_s(\widetilde Z_s,\widetilde Z_s)
	\Big]\,ds \nonumber\\
	&\quad
	-\int_\tau^T Z_s\,dW_s
	-\int_\tau^T \widetilde Z_s\,d\widetilde W_s .
\end{align}

By Proposition \ref{change_measure},
\[
B_s=W_s,\qquad
d\widetilde B_s=d\widetilde W_s+
\big(\widetilde\sigma^{+} b^{\alpha_s^*}\big)(s,X_s)\,ds
\]
under $\bQ$. Substituting $dW_s=dB_s$ and
$d\widetilde W_s=d\widetilde B_s-\big(\widetilde\sigma^{+} b^{\alpha_s^*}\big)(s,X_s)\,ds$
in \eqref{BSDE-value-under-P} cancels the additional drift term and yields
\[
Y_\tau
=G(X_T)+\int_\tau^T F_s(\widetilde Z_s,\widetilde Z_s)\,ds
-\int_\tau^T Z_s\,dB_s
-\int_\tau^T \widetilde Z_s\,d\widetilde B_s,
\]
for every $\tau\in[t,T]$, $\bQ$-almost surely. This is precisely \eqref{BSDE_v_valuefnc}.
\end{proof}

Here, we note that in the above, the processes $Y^n$, $Z^n$, and $\widetilde{Z}^n$ are parameterized by the initial condition $(t,x)$. For notational convenience, this dependence is omitted throughout, and we write $Y^n$, $Z^n$, and $\widetilde{Z}^n$ in place of $Y^{n;t,x}$, $Z^{n;t,x}$, and $\widetilde{Z}^{n;t,x}$. The same convention is adopted for $Y$, $Z$, and $\widetilde{Z}$.

To prove the convergence of the sequence $(v^n)_{n\in\bN}$ generated by the linear BSPDE \eqref{SHJB-Algorithm} toward the value function $v$, it suffices, in view of Proposition \ref{prop:linearBSPDE2BSDE} and \ref{prop:semilinearSHJB2BSDE}, to prove the convergence of the corresponding BSDE solutions $(Y^{n})_{n\in\bN}$ to $Y$, where $Y^{n}$ and $Y$ satisfy \eqref{BSDE_v^n_alg} and \eqref{BSDE_v_valuefnc}, respectively.
 We therefore begin by establishing the well-posedness of the BSDE \eqref{BSDE_v^n_alg}.  

Recall that the driver of the BSDEs \eqref{BSDE_v^n_alg} and \eqref{BSDE_v_valuefnc} is given by
\begin{equation}
	F_s(\tilde{z}_1,\tilde{z}_2):= \big(\widetilde{\sigma}^{+}b^{\boldsymbol{\alpha}(s,X_s,\tilde{z}_1)}\big)^\cT (s,X_s)\tilde{z}_2 +f^{\boldsymbol{\alpha}(s,X_s,\tilde{z}_1)}(s,X_s),
\end{equation}
 for $s\in[0,T],\;\tilde{z}_1,\tilde{z}_2\in \bR^m$. By Assumptions \ref{ass:meas_Lipschitz_LinearGrowth} and \ref{ass:continuity_bounded_integrability}, the process $F_{\cdot}(\tilde{z}_1,\tilde{z}_2)$ 
 takes values in $\bR$ and is $\cP^W-$measurable for every fixed  $\tilde{z}_1,\tilde{z}_2\in\bR^m$. 
 Moreover, since the control set $A$ is compact and $\boldsymbol{\alpha}(s,x,\tilde{z})$ is valued in $ A$, there exists a constant
$C_A>0$ such that
\[
|\boldsymbol{\alpha}(s,x,\tilde{z})|\le C_A,
\quad
(s,x,\tilde{z})\in[0,T]\times\bR^d\times\bR^m.
\]
Hence, using the linear growth condition of $f$ and Assumption \ref{ass:control_lipschitz_bound}, we have for all $s\in[0,T]$, $z_1,z_2,z'_2\in\bR^m$, $\bP\text{-a.s.}$,
 \begin{align*}
    |F_s(0,0)|&=|f(s,X_s,\boldsymbol{\alpha}(s,X_s,0))| \le K (1+C_A+|X_s|),
    \\
	|F_s(z_1,z_2)| & \le |F_s(0,0)|+ \theta |z_1|+\hat K|z_2|, 
    \\
    |F_s(z_1,z_2) -F_s(z_1,z'_2)|  & \le \hat K |z_2-z'_2|,
 \end{align*}
 where the constant $\hat K$ depends on $K$ and $\gamma$.

\begin{lem}\label{lem:wellposed_BSDE_vn}
Suppose that Assumptions \ref{ass:meas_Lipschitz_LinearGrowth}--\ref{ass:feedback_control_lipschitz} hold.
 Let $\widetilde{Z}^{n-1}\in \cL_{\cP}^2(\bQ,\bR^m,[t,T])$, for some $n\in \bN$. Then the BSDE \eqref{BSDE_v^n_alg} admits a unique adapted solution $(Y^n,Z^n,\widetilde{Z}^n)\in \big(\cS_{\cP}^2(\bQ,\bR,[t,T])\times \cL^2_{\cP}(\bQ,\bR^m,[t,T])\times \cL^2_{\cP}(\bQ,\bR^m,[t,T])\big)$.
 \end{lem}

In view of Proposition \ref{prop:semilinearSHJB2BSDE} and the boundedness of the spatial gradient of the value function $v$ in \eqref{value_fnc_gradient_bound},
we have $|\widetilde{Z}| \le M$ a.e., for some constant $M$ depending on $K$ and $T$. Notice that for $z_1,\tilde z, \tilde z'\in \bR^m$ with $|\tilde z'|\le M$, we have
\begin{align}
    &|F_s(z_1, \tilde z) -F_s(\tilde z', \tilde z')| 
    \notag \\
    &\le |(\widetilde{\sigma}^{+}b^{\boldsymbol{\alpha}(s,X_s,z_1)})^\cT (s,X_s)\tilde z 
        - (\widetilde{\sigma}^{+}b^{\boldsymbol{\alpha}(s,X_s,\tilde z')})^\cT (s,X_s)\tilde z'| + |f^{\boldsymbol{\alpha}(s,X_s,z_1)}(s,X_s) - f^{\boldsymbol{\alpha}(s,X_s,\tilde z')}(s,X_s)| 
    \notag\\
    &\le |(\widetilde{\sigma}^{+} b^{\boldsymbol{\alpha}(s,X_s,z_1)})(s,X_s)| \cdot |\tilde z - \tilde z'| + |b^{\boldsymbol{\alpha}(s,X_s,z_1)} - b^{\boldsymbol{\alpha}(s,X_s,\tilde z')}| \cdot |(\widetilde{\sigma}^{+})^\cT (s,X_s)\tilde z'| + \theta |z_1 - \tilde z'|
    \notag\\
    &\le \tilde K |\tilde z - \tilde z'| + \tilde K|z_1 - \tilde z'|,\quad \bQ \text{-a.s.}, \label{Lipschitz_driver}
\end{align}    
where $\tilde K$ depends on $\gamma$, $K$, $\theta$, and $M$. The above Lipschitz property of the driver $F$ will be used to establish the uniqueness of the solution $(Y,Z,\widetilde{Z})$ with bounded $\widetilde Z$ for the BSDE \eqref{BSDE_v_valuefnc}, 
while the existence of such a solution has been guaranteed by Proposition \ref{prop:semilinearSHJB2BSDE}.
 \begin{lem}\label{lem:wellposed_BSDE_v_valuefnc}
     Let Assumptions \ref{ass:meas_Lipschitz_LinearGrowth}--\ref{ass:feedback_control_lipschitz} hold.
      Then the BSDE \eqref{BSDE_v_valuefnc} admits a unique adapted solution $(Y,Z,\widetilde{Z})\in \big(\cS_{\cP}^2(\bQ,\bR,[t,T])\times \cL^2_{\cP}(\bQ,\bR^m,[t,T])\times \cL^2_{\cP}(\bQ,\bR^m,[t,T]) \big)$ 
      with   \\
      $\esssup_{(\omega,s)\in\Omega\times[t,T]}|\widetilde Z_s(\omega)|<\infty$.
 \end{lem}

The preceding well-posedness results follow from calculations as in \cite[Theorem 4.2, Chapter 1]{JMaJYong2007FBSDEAndAppl} and we omit the proof.

\subsection{Convergence Estimates}

We next derive an a priori estimate for the solution of the iterative BSDE \eqref{BSDE_v^n_alg}. 

 %%%%%%%%%%%%%%%%%%%%%%%%%%%%%%%%%%%%

 \begin{lem}\label{lem:priori_bsde_Yn}
     Suppose that Assumptions \ref{ass:meas_Lipschitz_LinearGrowth}--\ref{ass:feedback_control_lipschitz} hold. Let $\widetilde Z^{0} \in \mathcal{L}^2_{\cP}(\bQ,\mathbb{R}^m,[t,T])$, and let $(Y^n, Z^n, \widetilde{Z}^n)$ denote the unique adapted solution to the BSDE \eqref{BSDE_v^n_alg} for $n\in\bN$. 
      Then there exist constants \(q\in(0,1/2)\), \(\lambda>1\), and \(C>0\), depending only on \(K,\theta, \gamma\), and \(T\), such that
     \begin{align}
    &\bE^\bQ\left[\sup_{\tau\in[t,T]}|Y^n_{\tau}|^2 \right] +\bE^\bQ \int_t^T e^{\lambda s} |Y^n_s|^2ds +\bE^\bQ \int_t^T e^{\lambda s}\left(|Z^n_s|^2+|\widetilde{Z}_s^n|^2\right)ds 
    \nonumber\\
    & \leq C\bE^\bQ  \Bigg[e^{\lambda T}|G(X_T)|^2  + \int_t^T e^{\lambda s} |F_s(0,0)|^2 ds \Bigg] +Cq^{n-1}\bE^\bQ \int_t^T e^{\lambda s}|\widetilde{Z}^0_s|^2ds.
\end{align}
 \end{lem}

\begin{proof}
    Applying It\^o's formula to $e^{\lambda s}|Y^n_s|^2$, with $\lambda >0$ to be chosen later, gives, for every $\tau\in[t,T]$,
    \begin{align}
        e^{\lambda \tau}|Y^n_{\tau}|^2 &+ \int_{\tau}^T e^{\lambda s}\Big(|Z^n_s|^2+|\widetilde{Z}_s^n|^2\Big)ds = e^{\lambda T}|G(X_T)|^2 + 2\int_{\tau}^T e^{\lambda s} F_s(\widetilde{Z}^{n-1}_s, \widetilde{Z}^n_s)Y^n_s ds \nonumber\\
        & -\int_{\tau}^T \lambda e^{\lambda s}|Y^n_s|^2 ds -2\int_{\tau}^T e^{\lambda s} Z^n_s Y^n_s dB_s -2\int_{\tau}^T e^{\lambda s }\widetilde{Z}^n_s Y^n_s d\widetilde{B}_s,
    \end{align}
    $\bQ$-almost surely. Taking expectations and using linear growth for the driver $F$, we obtain
\begin{align}
        & \bE^\bQ \Big[e^{\lambda \tau}|Y^n_{\tau}|^2 \Big]+ \bE^\bQ \int_{\tau}^T e^{\lambda s}\Big(|Z^n_s|^2+|\widetilde{Z}_s^n|^2\Big)ds +\bE^\bQ \int_{\tau}^T \lambda e^{\lambda s}|Y^n_s|^2 ds \nonumber\\
        &  \leq \bE^\bQ \Big[e^{\lambda T}|G(X_T)|^2 \Big] + 2\bE^\bQ \int_{\tau}^T e^{\lambda s} |F_s(\widetilde{Z}^{n-1}_s, \widetilde{Z}^n_s)|\cdot |Y^n_s| ds  \nonumber\\
        & \leq \bE^\bQ \Big[e^{\lambda T}|G(X_T)|^2 \Big] + 2\bE^\bQ  \int_{\tau}^T e^{\lambda s} \Big(|F_s(0,0)| + \theta |\widetilde{Z}^{n-1}_s|+\hat K|\widetilde{Z}^n_s| \Big) |Y^n_s| ds.
    \end{align}
    Applying Young's inequality to the right-hand side, we obtain, for any $\eps>0$,
\begin{align}
        & \bE^\bQ \Big[e^{\lambda \tau}|Y^n_{\tau}|^2 \Big]+ \bE^\bQ \int_{\tau}^T e^{\lambda s}\left(|Z^n_s|^2+\left(1-\eps \hat K\right)|\widetilde{Z}_s^n|^2\right)ds +\big(\lambda-(1+{\theta}+\hat K)\frac{1}{\eps}\big)\bE^\bQ \int_{\tau}^T e^{\lambda s}|Y^n_s|^2 ds \nonumber\\
        & \leq \bE^\bQ \Big[e^{\lambda T}|G(X_T)|^2 \Big] 
        + \epsilon\bE^\bQ  \int_{\tau}^T e^{\lambda s} |F_s(0,0)|^2 ds 
        + \eps {\theta}\bE^\bQ  \int_{\tau}^T e^{\lambda s} |\widetilde{Z}^{n-1}_s|^2 ds.
    \end{align}
We now choose $\lambda>1$ sufficiently large so that 
$\widetilde{q}:=\max \left\{ \frac{(1+{\theta}+\hat K)\hat K}{\lambda -1} , \frac{(1+{\theta}+\hat K){\theta}}{\lambda -1}\right\}\in (0,1/3)$ and fix this value of $\lambda$. Next, choose $\eps$ so that 
$\lambda-(1+{\theta}+\hat K)\frac{1}{\eps}=1$. 
Since \[ \eps \hat K\le\widetilde q, \qquad \eps \theta\le\widetilde q, \] it follows that
\begin{align}
        & \bE^\bQ \Big[e^{\lambda \tau}|Y^n_{\tau}|^2 \Big]+ \bE^\bQ \int_{\tau}^T e^{\lambda s}\left(|Z^n_s|^2+\left(1-\widetilde{q}\right)|\widetilde{Z}_s^n|^2\right)ds +\bE^\bQ \int_{\tau}^T e^{\lambda s}|Y^n_s|^2 ds \nonumber\\
        & \leq \bE^\bQ \Big[e^{\lambda T}|G(X_T)|^2 \Big] + \frac{1+{\theta}+\hat K}{\lambda -1}\bE^\bQ  \int_{\tau}^T e^{\lambda s} |F_s(0,0)|^2 ds + \widetilde{q} \bE^\bQ  \int_{\tau}^T e^{\lambda s} |\widetilde{Z}^{n-1}_s|^2 ds.
    \end{align}
    
Since $\widetilde{q}\in (0,1/3)$, we have $(1-\widetilde{q})\in (2/3,1)$ and $q:=\frac{\widetilde{q}}{1-\widetilde{q}}\in (0,1/2) $. Dividing both sides of the preceding estimate by $1-\widetilde{q}$ yields
\begin{align}\label{est:priori_seq_1_1}
        & \bE^\bQ \Big[e^{\lambda \tau}|Y^n_{\tau}|^2 \Big]+ \bE^\bQ \int_{\tau}^T e^{\lambda s}\left(|Z^n_s|^2+|\widetilde{Z}_s^n|^2\right)ds +\bE^\bQ \int_{\tau}^T e^{\lambda s}|Y^n_s|^2 ds \nonumber\\
        & \leq \frac{3}{2}\bE^\bQ \Bigg[e^{\lambda T}|G(X_T)|^2  +\frac{1}{3\max\{\hat K,{\theta}\}}  \int_{\tau}^T e^{\lambda s} |F_s(0,0)|^2 ds \Bigg]+ q \bE^\bQ  \int_{\tau}^T e^{\lambda s} |\widetilde{Z}^{n-1}_s|^2 ds
        \end{align}
        Iterating the estimate \eqref{est:priori_seq_1_1} for the last term on the right-hand side yields
$$\sum_{k=0}^{n-1}q^k =\frac{1-q^n}{1-q}\le\frac{1}{1-q}<2,$$
since \(q\in(0,\frac12)\). Consequently,
        % Now using the above estimate for the last term on the right-hand side of \eqref{est:priori_seq_1_1} recurrently, we will have a geometric series with ratio $q\in(0,1/2)$. Finally, using the upper bound of the partial sum of the associated geometric series, we obtain the following estimate
        \begin{align}\label{est:priori_seq_1}
        & \bE^\bQ \Big[e^{\lambda \tau}|Y^n_{\tau}|^2 \Big]+ \bE^\bQ \int_{\tau}^T e^{\lambda s}\left(|Z^n_s|^2+|\widetilde{Z}_s^n|^2\right)ds +\bE^\bQ \int_{\tau}^T e^{\lambda s}|Y^n_s|^2 ds \nonumber\\
        & \leq 3\bE^\bQ \Bigg[e^{\lambda T}|G(X_T)|^2  +\frac{1}{3\max\{\hat K,{\theta}\}} \int_{\tau}^T e^{\lambda s} |F_s(0,0)|^2 ds \Bigg]+ q^{n} \bE^\bQ  \int_{\tau}^T e^{\lambda s} |\widetilde{Z}^{0}_s|^2 ds.
    \end{align}

On the other hand, from \eqref{BSDE_v^n_alg}, for every $\tau\in [t,T]$,
\begin{align}
    |Y^n_{\tau}| \leq |G(X_T)|+\Big|\int_{\tau}^T F_s(\widetilde{Z}^{n-1}_s, \widetilde{Z}^n_s)ds\Big| +\Big|\int_{\tau}^T Z^n_s dB_s\Big|+\Big|\int_{\tau}^T \widetilde{Z}^n_s d\widetilde{B}_s\Big|\quad \text{$\bQ$-a.s.}
\end{align}
Hence, using \((a_1+\cdots+a_4)^2\le 4(a_1^2+\cdots+a_4^2)\) and Cauchy's inequality, it holds that $\bQ$-a.s.,
\begin{align*}
    |Y^n_{\tau}|^2 \leq 4|G(X_T)|^2+4(T-\tau)\int_{\tau}^T \big|F_s(\widetilde{Z}^{n-1}_s, \widetilde{Z}^n_s)\big|^2ds +4\Big|\int_{\tau}^T Z^n_s dB_s\Big|^2+4\Big|\int_{\tau}^T \widetilde{Z}^n_s d\widetilde{B}_s\Big|^2.
\end{align*}
Taking the supremum over $\tau\in[t,T]$ and then expectations yields
\begin{align}\label{eqref_1}
    \bE^\bQ \left[\sup_{\tau\in[t,T]}|Y^n_{\tau}|^2 \right]&\leq 4\bE^\bQ |G(X_T)|^2+4(T-t)\bE^\bQ \int_t^T \big|F_s(\widetilde{Z}^{n-1}_s, \widetilde{Z}^n_s)\big|^2ds \nonumber\\
    &+4\bE^\bQ \left[ \sup_{\tau\in[t,T]}\Big|\int_{\tau}^T Z^n_s dB_s\Big|^2\right]+4\bE^\bQ  \left[\sup_{\tau\in [t,T]}\Big|\int_{\tau}^T \widetilde{Z}^n_s d\widetilde{B}_s\Big|^2\right].
\end{align}
By Doob's inequality and It\^o's isometry, we have
\begin{align}
    \bE^\bQ \left[ \sup_{\tau\in [t,T]} \Big|\int_{\tau}^T \widetilde{Z}^n_s d\widetilde{B}_s\Big|^2\right]\leq 4\bE^\bQ  \left[\Big|\int_t^T \widetilde{Z}^n_s d\widetilde{B}_s\Big|^2 \right] 
    = 4\bE^\bQ  \int_t^T |\widetilde{Z}^n_s|^2 ds.
\end{align}
and similarly,
\begin{align}
    \bE^\bQ \left[ \sup_{\tau\in [t,T]} \Big|\int_{\tau}^T Z^n_s dB_s\Big|^2\right]
    \leq
    4\bE^\bQ  \int_t^T |Z^n_s|^2 ds.
\end{align}
Using these estimates, the linear growth of the driver $F$, and relation \eqref{est:priori_seq_1}, it follows from \eqref{eqref_1} that
\begin{align*}
    &\bE^\bQ \left[\sup_{\tau\in[t,T]}|Y^n_{\tau}|^2 \right] \nonumber \\
    &\leq 4\bE^\bQ |G(X_T)|^2+4(T-t)\bE^\bQ \int_t^T \big|F_s(\widetilde{Z}^{n-1}_s, \widetilde{Z}^n_s)\big|^2ds  
      +16\bE^\bQ  \int_t^T \left(|Z^n_s|^2+|\widetilde{Z}^n_s|^2\right) ds\nonumber\\
    & \leq C\bE^\bQ  \Bigg[e^{\lambda T}|G(X_T)|^2  + \int_t^T e^{\lambda s} |F_s(0,0)|^2 ds \Bigg] +Cq^{n-1}\bE^\bQ \int_t^T e^{\lambda s} |\widetilde{Z}^0_s|^2ds,
\end{align*}
which together with \eqref{est:priori_seq_1}   yields the desired estimate.
\end{proof}

We proceed to establish the convergence of the sequence $(Y^n,Z^n,\widetilde{Z}^n)_{n\in\bN}$, defined as the successive solutions of the iterative BSDE \eqref{BSDE_v^n_alg}, toward the unique adapted solution $(Y,Z,\widetilde{Z})$ of the nonlinear BSDE \eqref{BSDE_v_valuefnc}.

\begin{thm}\label{thm:main_convergence_bsde}
    Suppose that Assumptions \ref{ass:meas_Lipschitz_LinearGrowth}--\ref{ass:feedback_control_lipschitz} hold. 
    Let $(Y,Z, \widetilde{Z})\in \cS_{\cP}^2(\bQ,\bR,[t,T])\times \cL_{\cP}^2(\bQ,\bR^m,[t,T])\times \cL_{\cP}^2(\bQ,\bR^m,[t,T])$ be the unique adapted solution of the BSDE \eqref{BSDE_v_valuefnc} in Lemma \ref{lem:wellposed_BSDE_v_valuefnc}. 
    Then the sequence $(Y^n,Z^n, \widetilde{Z}^n)_{n\in\bN}$, where $(Y^n,Z^n, \widetilde{Z}^n)$ is the unique solution of \eqref{BSDE_v^n_alg}, converges to $(Y,Z, \widetilde{Z})$ in $\cS_{\cP}^2(\bQ,\bR,[t,T])\times \cL_{\cP}^2(\bQ,\bR^m,[t,T])\times \cL_{\cP}^2(\bQ,\bR^m,[t,T])$. 
    More precisely, there exist constants $C>0$ and $q\in \Big(0, \frac{1}{2}\Big)$ depending only on $K$,$\gamma$, $\theta$, and $T$, such that
    \begin{align}
         \bE^{\bQ} \Big[ \sup_{\tau\in[t,T]}|Y^n_{\tau}-Y_{\tau}|^2+\int_t^T\left(|\widetilde{Z}^n_s-\widetilde{Z}_s|^2+|Z^n_s-Z_s|^2\right)ds\Big]
         \leq C q^{n-1} \bE^{\bQ} \int_t^T|\widetilde{Z}^0_s-\widetilde{Z}_s|^2ds \rightarrow 0,
    \end{align}
    as $n$ tends to $\infty$. 
\end{thm}

%%%%%%%%

\begin{proof}
    We compare the iterative BSDE directly with the nonlinear BSDE. Set
    \[
        \Delta Y^n:=Y^n-Y,\qquad
        \Delta Z^n:=Z^n-Z,\qquad
        \Delta \widetilde{Z}^n:=\widetilde{Z}^n-\widetilde{Z},
    \]
    and
    \[
        \Delta F^n_s
        :=F_s(\widetilde{Z}^{n-1}_s,\widetilde{Z}^n_s)
        -F_s(\widetilde{Z}_s,\widetilde{Z}_s).
    \]
    Since $\widetilde{Z}$ is essentially bounded, we may apply \eqref{Lipschitz_driver} with
    $z_1=\widetilde{Z}^{n-1}_s$, $\tilde z=\widetilde{Z}^n_s$, and
    $\tilde z'=\widetilde{Z}_s$, and we have
    \[
        |\Delta F^n_s|
        \leq \tilde K|\Delta\widetilde{Z}^n_s|
        +\tilde K |\Delta\widetilde{Z}^{n-1}_s|,
        \qquad s\in[t,T],\quad \bQ\text{-a.s.}
    \]
    Applying It\^o's formula to $e^{\lambda s}|\Delta Y^n_s|^2$ and taking conditional expectations with respect to $\cF_\tau$, for $\tau\in[t,T]$, yield
    \begin{align}
        &\bE_{\cF_\tau}^{\bQ}\bigg[
        e^{\lambda \tau}|\Delta Y^n_\tau|^2
        +\lambda\int_\tau^T e^{\lambda s}|\Delta Y^n_s|^2ds
        +\int_\tau^T e^{\lambda s}\Big(|\Delta Z^n_s|^2+|\Delta\widetilde{Z}^n_s|^2\Big)ds
        \bigg]\nonumber\\
        &\leq
        2\bE_{\cF_\tau}^{\bQ}\int_\tau^T e^{\lambda s}|\Delta F^n_s|\cdot|\Delta Y^n_s|ds \nonumber\\
        &\leq
        2\bE_{\cF_\tau}^{\bQ}\int_\tau^T e^{\lambda s}
        \Big(\tilde K|\Delta\widetilde{Z}^n_s|
        +\tilde K|\Delta\widetilde{Z}^{n-1}_s|\Big)
        |\Delta Y^n_s|ds .
    \end{align}
    By Young's inequality, for every $\varepsilon>0$,
    \begin{align}
        &\bE_{\cF_\tau}^{\bQ}\bigg[
        e^{\lambda \tau}|\Delta Y^n_\tau|^2
        +\Big(\lambda-\frac{2 \tilde K}{\varepsilon}\Big)
        \int_\tau^T \!\! e^{\lambda s}|\Delta Y^n_s|^2ds
        +\int_\tau^T\!\! e^{\lambda s}|\Delta Z^n_s|^2ds 
        % +\int_\tau^T e^{\lambda s}|\Delta\widetilde{Z}^n_s|^2ds
        % \bigg]
        % \nonumber\\
        % &\hspace{3.5cm}
        +\left(1-\varepsilon \tilde K\right)
        \int_\tau^T e^{\lambda s}|\Delta\widetilde{Z}^n_s|^2ds
        \bigg]\nonumber\\
        &\leq
        \varepsilon \tilde K\,
        \bE_{\cF_\tau}^{\bQ}\int_\tau^T e^{\lambda s}
        |\Delta\widetilde{Z}^{n-1}_s|^2ds .
    \end{align}
    Choose $\lambda>1$ so large that
    \[
        \bar q:=
        \frac{2\tilde K^2}{\lambda-1} \in(0,1/3),
    \]
    and then choose $\varepsilon$ such that
    $\lambda- 2 \tilde K /\varepsilon=1$. Then
    $\varepsilon\tilde K\leq \bar q$. Hence,
    \begin{align}
        &\bE_{\cF_\tau}^{\bQ}\bigg[
        e^{\lambda \tau}|\Delta Y^n_\tau|^2
        +\int_\tau^T e^{\lambda s}|\Delta Y^n_s|^2ds
        +\int_\tau^T e^{\lambda s}|\Delta Z^n_s|^2ds
        +(1-\bar q)\int_\tau^T e^{\lambda s}|\Delta\widetilde{Z}^n_s|^2ds
        \bigg]\nonumber\\
        &\leq
        \bar q\,
        \bE_{\cF_\tau}^{\bQ}\int_\tau^T e^{\lambda s}
        |\Delta\widetilde{Z}^{n-1}_s|^2ds .
    \end{align}
    Since $1-\bar q\in(2/3,1)$, setting $q:=\bar q/(1-\bar q)\in(0,1/2)$ gives
    \begin{align}\label{est:direct_step_contraction}
        &\bE_{\cF_\tau}^{\bQ}\bigg[
        e^{\lambda \tau}|\Delta Y^n_\tau|^2
        +\int_\tau^T e^{\lambda s}|\Delta Y^n_s|^2ds
        +\int_\tau^T e^{\lambda s}
        \Big(|\Delta Z^n_s|^2+|\Delta\widetilde{Z}^n_s|^2\Big)ds
        \bigg]\nonumber\\
        &\leq
        q\,\bE_{\cF_\tau}^{\bQ}\int_\tau^T e^{\lambda s}
        |\Delta\widetilde{Z}^{n-1}_s|^2ds .
    \end{align}
    Iterating \eqref{est:direct_step_contraction} yields, for every $n\in\bN$,
    \begin{align}\label{est:direct_iterated_contraction}
        &\bE_{\cF_\tau}^{\bQ}\bigg[
        e^{\lambda \tau}|\Delta Y^n_\tau|^2
        +\int_\tau^T e^{\lambda s}|\Delta Y^n_s|^2ds
        +\int_\tau^T e^{\lambda s}
        \Big(|\Delta Z^n_s|^2+|\Delta\widetilde{Z}^n_s|^2\Big)ds
        \bigg]\nonumber\\
        &\leq
        q^{n-1}\,\bE_{\cF_\tau}^{\bQ}\int_\tau^T e^{\lambda s}
        |\widetilde{Z}^{0}_s-\widetilde{Z}_s|^2ds .
    \end{align}
    It remains only to upgrade the convergence of $Y^n$ in $\cS^2$. From the difference of \eqref{BSDE_v^n_alg} and \eqref{BSDE_v_valuefnc}, for every $\tau\in[t,T]$,
    \[
        \Delta Y^n_\tau
        =\int_\tau^T \Delta F^n_sds
        -\int_\tau^T \Delta Z^n_sdB_s
        -\int_\tau^T \Delta\widetilde{Z}^n_sd\widetilde{B}_s .
    \]
    Therefore, by Cauchy's inequality, Doob's inequality, and It\^o's isometry, for a constant $C>0$ independent of $n$,
    \begin{align}
        \bE^{\bQ}\bigg[\sup_{\tau\in[t,T]}|\Delta Y^n_\tau|^2\bigg]
        &\leq C\bE^{\bQ}\int_t^T |\Delta F^n_s|^2ds
        +C\bE^{\bQ}\int_t^T
        \Big(|\Delta Z^n_s|^2+|\Delta\widetilde{Z}^n_s|^2\Big)ds \nonumber\\
        &\leq C\bE^{\bQ}\int_t^T
        \Big(|\Delta\widetilde{Z}^{n-1}_s|^2
        +|\Delta\widetilde{Z}^{n}_s|^2
        +|\Delta Z^n_s|^2\Big)ds,
    \end{align}
    where the last step uses \eqref{Lipschitz_driver}. This  combined with \eqref{est:direct_iterated_contraction} yields the desired limit.
\end{proof}

%%%%%%%%%%%
%%%%%%%%%%%%%%
\subsection{Proof of Theorem \ref{thm:convergence_value_function} and Monotone Improvement}
Having established the required BSDE representations and convergence estimates, we are now in a position to prove the main convergence result, Theorem \ref{thm:convergence_value_function}.

\begin{proof}[Proof of Theorem \ref{thm:convergence_value_function}]
    Let $(Y,Z,\widetilde{Z})$ be the unique adapted solution to the nonlinear BSDE \eqref{BSDE_v_valuefnc} in Lemma \ref{lem:wellposed_BSDE_v_valuefnc}, 
    and let $(Y^n,Z^n,\widetilde{Z}^n)$ be the unique adapted solution to the iterative BSDE \eqref{BSDE_v^n_alg} in Lemma \ref{lem:wellposed_BSDE_vn}. 
    Proceeding exactly as in the proof of Theorem \ref{thm:main_convergence_bsde},  we have by \eqref{est:direct_iterated_contraction} that
    \begin{equation}
        |Y^n_{\tau}-Y_{\tau}|^2 \leq q^n\;\bE_{\cF_{\tau}}^{\bQ}\int_{\tau}^T e^{\lambda s} |\widetilde{Z}_s^{0}-\widetilde{Z}_s|^2ds.
    \end{equation}
    Taking $\tau=t$ and then taking expectation under $\bP$, we obtain
     \begin{align}
        \bE^{\bP}\bigg[  |Y^n_{t}-Y_t|^2\bigg]
        \leq  q^n\;\bE^{\bP}\left[ \bE_{\cF_{t}}^{\bQ}\int_{t}^T e^{\lambda s} |\widetilde{Z}_s^{0}-\widetilde{Z}_s|^2ds\right].
    \end{align}
  Using the representations
\[
Y_{\tau}^n=v^n(\tau,X^{t,x}_\tau),\qquad
Y_{\tau}=v(\tau,X^{t,x}_{\tau}), \quad \tau\in[t,T],
\]
and
\[
\widetilde Z_s^0=(\widetilde\sigma^\cT  D_xv^0)(s,X_s^{t,x}),
\qquad
\widetilde Z_s=(\widetilde\sigma^\cT  D_xv)(s,X_s^{t,x}), \quad s\in[t,T],
\]
we obtain for a.e. $x\in\bR^d$,
    \begin{align}\label{est_temp_errvalue_funct}
&\bE^{\bP}\bigg[ |v^n(t,x)-v(t,x)|^2\bigg]
\nonumber\\
&\leq  q^{n} 
\bE^{\bP}\left[
     \bE_{\cF_{t}}^{\bQ}\int_{t}^T e^{\lambda s} \Big|(\widetilde{\sigma}^\cT D_xv^0)(s,X_s^{t,x})
     -(\widetilde{\sigma}^\cT D_xv)(s,X_s^{t,x})\Big|^2 ds\right]\nonumber\\
    &\leq  2q^{n} \bE^{\bP}\left[ 
        \bE_{\cF_{t}}^{\bQ}\int_{t}^T 
        e^{\lambda s} \Big|\widetilde{\sigma} (s,X_s^{t,x})\Big|^2 
        \Big(\big|D_xv^0(s,X_s^{t,x})\big|^2+\big|D_xv(s,X_s^{t,x})\big|^2\Big) ds\right].
    \end{align}
By Assumption \ref{ass:continuity_bounded_integrability}, $\widetilde{\sigma}$ is uniformly bounded by $K$, $\bP$-almost surely. 
Recall in \eqref{value_fnc_gradient_bound}, whenever the corresponding spatial gradients exist, 
$$|D_xv(t,x)| + |D_xv^0(t,x)|\leq M, \quad\bP\text{-a.s.,}$$ 
where $M>0$ is a constant depending on $ K$, and $T$.
Since $\bP$ and $\bQ$ are equivalent, these bounds also hold $\bQ$-almost surely. Hence, from \eqref{est_temp_errvalue_funct},
 it follows straightforwardly that for each $t\in[0,T]$,
\begin{align*}
         \bE^{\bP}\bigg[|v^n(t,x)-v(t,x)|^2\bigg]
         \leq \frac{e^{\lambda T}}{\lambda}K^2M^2 q^n
         \leq C q^{n}, \quad \text{for a.e. } x\in\bR^d,
    \end{align*}
  where the constant $C>0$ depends on $\theta, \gamma, K$, and $T$.
\end{proof}
 
We conclude this section by establishing the monotone improvement property of the proposed policy iteration scheme. More precisely, we show that $v^n$, the approximating sequence of the value function, is monotone non-decreasing.
\begin{prop} \label{prop:monotonicity_value_function}
    Suppose that Assumptions \ref{ass:meas_Lipschitz_LinearGrowth}--\ref{ass:feedback_control_lipschitz} hold.  
    Let $(v^n(t,x),\psi^n(t,x))$ and $(v^{n+1}(t,x),\psi^{n+1}(t,x))$ denote the solutions of the 
    linear BSPDE \eqref{SHJB-Algorithm} for two consecutive iterations. 
    Then for each $t\in[0,T]$,
  $$v^n(t,x)\leq v^{n+1}(t,x),\quad \bP\otimes dx\text{-a.e. } (\omega, x)\in \Omega\times \bR^d.$$ 
 \end{prop}
\begin{proof}  Consider the BSDE \eqref{BSDE_v^n_alg} at two consecutive iterations:
    \begin{align}
        Y^n_t &= G(X_T) + \int_t^T F_s(\widetilde{Z}^{n-1}_s, \widetilde{Z}^n_s )ds - \int_t^T Z^n_s dB_s - \int_t^T \widetilde{Z}^n_s d\widetilde{B}_s, \label{BSDE_Y_n} \\
        Y^{n+1}_t &= G(X_T) + \int_t^T F_s(\widetilde{Z}^{n}_s, \widetilde{Z}^{n+1}_s )ds - \int_t^T Z^{n+1}_s dB_s - \int_t^T \widetilde{Z}^{n+1}_s d\widetilde{B}_s. \label{BSDE_Y_n+1}
    \end{align}
    For every $s\in[t,T]$ define the functions 
    \begin{equation*}
        \phi^1_s(z) := F_s(\widetilde{Z}_s^{n-1},z) \quad\text{and}\quad \phi^2_s(z):= F_s(\widetilde{Z}^n_s, z), \qquad z\in \bR^m,
    \end{equation*}
    so that the drivers of BSDEs \eqref{BSDE_Y_n} and \eqref{BSDE_Y_n+1} differ only through the first argument of $F$. With these definitions, BSDEs \eqref{BSDE_Y_n} and \eqref{BSDE_Y_n+1} become 
 \begin{align}
        Y^n_t &= G(X_T) + \int_t^T \phi^1_s( \widetilde{Z}^n_s )ds - \int_t^T Z^n_s dB_s - \int_t^T \widetilde{Z}^n_s d\widetilde{B}_s, \label{BSDE_Y_n_new} \\
        Y^{n+1}_t &= G(X_T) + \int_t^T \phi^2_s( \widetilde{Z}^{n+1}_s )ds - \int_t^T Z^{n+1}_s dB_s - \int_t^T \widetilde{Z}^{n+1}_s d\widetilde{B}_s. \label{BSDE_Y_n+1_new}
    \end{align}
    Moreover, for every $s\in[t,T]$,
    \begin{align}
        \phi^2_s(\widetilde{Z}^n_s) 
        = F_s(\widetilde{Z}^{n}_s, \widetilde{Z}^n_s) 
        & = \left(\widetilde{\sigma}^{+}b^{\boldsymbol{\alpha}(s,X_s, \widetilde{Z}^n_s)}\right)^{\cT}(s,X_s)\widetilde{Z}^n_s + f^{\boldsymbol{\alpha}(s,X_s, \widetilde{Z}^n_s)}(s,X_s)\nonumber\\
        % &  = \left(b^{a^{n+1}}D_xv^n\right)(s,X_s) + f^{a^{n+1}}(s,X_s)\nonumber\\
        & = \max_{a\in A} \Big(\left( (b^a)^\cT D_xv^n\right)(s,X_s)+f^a(s,X_s) \Big) \nonumber\\
        & \geq \left( (b^{a^{n}})^\cT D_xv^n\right)(s,X_s)+f^{a^n}(s,X_s) \nonumber\\
        & =F_s(\widetilde{Z}^{n-1}_s, \widetilde{Z}^n_s)\nonumber\\
        &=\phi_s^1(\widetilde{Z}^n_s).
    \end{align}
Since
\[
\phi_s^2(\widetilde Z_s^n)\ge
\phi_s^1(\widetilde Z_s^n),
\qquad s\in[t,T],\qquad \bQ\text{-a.s.,}
\] and
\(\phi^2(\widetilde Z^n)\in\cL_{\cP}^2(\bQ,\bR)\) with $\phi^2_s(\cdot)$ being Lipschitz continuous, the standard comparison theorem for BSDEs \cite[Theorem 6.2.2]{HPham2009ContstimeStoControlAndOptimizationFinancialAppl} implies that
\[
Y_s^n\le Y_s^{n+1},
\qquad s\in[t,T],\qquad \bQ\text{-a.s.}
\]
Recalling the relation $Y_s^n=v^n(s,X_s^{t,x})$ and $Y_s^{n+1}=v^{n+1}(s,X_s^{t,x}),$
we obtain
\[
v^n(s,X_s^{t,x})
\le
v^{n+1}(s,X_s^{t,x}),
\qquad
s\in[t,T],\qquad \bQ\text{-a.s.}
\]
Set $s=t$. Since the initial point \((t,x)\in[0,T]\times\mathbb R^d\) was arbitrary and \(\bP\) and \(\bQ\) are equivalent probability measures, it follows that for each $t\in [0,T]$,
\[
v^n(t,x)\le v^{n+1}(t,x),\qquad 
\bP\otimes dx\text{-a.e. } (\omega, x)\in \Omega\times \bR^d.
\]  
\end{proof}

\section{Conclusion}\label{sec:conclusion}

This paper develops and analyzes a policy-iteration scheme for non-Markovian stochastic optimal control problems through the associated SHJB equation, formulated as a semilinear BSPDE. Under appropriate regularity assumptions, the solution of the SHJB equation is identified with the value function of the underlying control problem, so that its approximation provides a natural route to solving the control problem itself.

To address the nonlinearity of the SHJB equation, we introduced an iterative policy-iteration framework in which the original semilinear BSPDE is replaced by a sequence of linear BSPDEs. We derive BSDE representations for both the semilinear equation and its linearized counterparts, and established a priori estimates for the solutions of such BSDEs.

The main result shows that the resulting sequence of approximations converges to the value function in the mean-square sense at an exponential rate. We also proved a monotone policy-improvement property, namely, that the sequence of iterates is nondecreasing when evaluated along the corresponding controlled state processes. Together, these results provide a rigorous theoretical foundation for the proposed policy-iteration scheme.

The proposed framework provides the first policy-iteration methodology for non-Markovian stochastic control problems governed by SHJB equations. It also provides a tractable approach to the approximation of semilinear SHJB equations. Since each iteration requires the solution of only a linear BSPDE, the framework is compatible with a broad range of existing numerical methods based on BSDE representations, including Monte Carlo regression, operator-splitting techniques, and deep learning-based approaches.

The present work has focused primarily on the theoretical properties of the proposed scheme. Several directions remain for future investigation, including the design and implementation of efficient numerical algorithms, the analysis of their stability and computational performance, particularly in high-dimensional settings, and the extension of the methodology to broader classes of stochastic control problems and fully nonlinear SHJB equations.

{\footnotesize
% Redefine thebibliography to hook in a tighter item separation
\let\oldthebibliography\thebibliography
\renewcommand\thebibliography[1]{%
    \oldthebibliography{#1}%
    \setlength{\itemsep}{0pt}%  <-- Reduces space between items
    \setlength{\parskip}{0pt}%  <-- Reduces paragraph space
}
\bibliographystyle{siam}
	\bibliography{reading_list}
    }
\end{document}